\documentclass[onefignum,onetabnum]{siamart171218}

\usepackage{graphicx}
\usepackage{booktabs}
\usepackage{amsmath}
\usepackage[english]{babel}
\usepackage[utf8]{inputenc}
\usepackage{lipsum}
\usepackage{amssymb}
\usepackage{multicol}
\usepackage[utf8]{inputenc}
\usepackage[english]{babel}
\usepackage{multirow}
\usepackage{hyperref}
\usepackage{cancel}
\usepackage{bbm}
\usepackage{algorithm}
\usepackage{algorithmic}

\newcommand{\bA}{{\bf A}}
\newcommand{\bB}{{\bf B}}

\newcommand{\bD}{{\bf D}}
\newcommand{\bE}{{\bf E}}

\newcommand{\bI}{{\bf I}}
\newcommand{\bJ}{{\bf J}}

\newcommand{\bL}{{\bf L}}
\newcommand{\bM}{{\bf M}}

\newcommand{\bW}{{\bf W}}

\newcommand{\bb}{{\bf b}}

\newcommand{\bd}{{\bf d}}
\newcommand{\be}{{\bf e}}
\newcommand{\bff}{{\bf f}}
\newcommand{\bg}{{\bf g}}

\newcommand{\br}{{\bf r}}
\newcommand{\bs}{{\bf s}}
\newcommand{\bt}{{\bf t}}

\newcommand{\bx}{{\bf x}}
\newcommand{\by}{{\bf y}}

\newcommand{\bepsilon}{{\boldsymbol{\epsilon}}}

\newsiamremark{remark}{Remark}

\newcommand{\TheTitle}{Convergence Analysis of a Variable Projection Method for Regularized Separable Nonlinear Inverse Problems }
\newcommand{\ShortTitle}{VarPro for separable nonlinear inverse problems}
\newcommand{\TheAuthors}{Malena I. Espa\~ nol and Gabriela Jeronimo}

\headers{\ShortTitle}{\TheAuthors}
\title{{\TheTitle}}

\author{
	Malena I. Espa\~ nol\thanks{School of Mathematical and Statistical Sciences, Arizona State University, Tempe, AZ, United States
		(\email{malena.espanol@asu.edu}).}
	\and
	Gabriela Jeronimo\thanks{Universidad de Buenos Aires. Facultad de Ciencias Exactas y Naturales. Departamento de Matemática. Buenos Aires, Argentina. \\ CONICET - Universidad de Buenos Aires. Instituto de Investigaciones Matem\'aticas ``Luis A. Santal\'o'' (IMAS). Buenos Aires, Argentina. (\email{jeronimo@dm.uba.ar}).}
}

\begin{document}

\maketitle

\begin{abstract} Variable projection methods prove highly efficient in solving separable nonlinear least squares problems by transforming them into a reduced nonlinear least squares problem, typically solvable via the Gauss-Newton method. When solving large-scale separable nonlinear inverse problems with general-form Tikhonov regularization, the computational demand for computing Jacobians in the Gauss-Newton method becomes very challenging. To mitigate this, iterative methods, specifically LSQR, can be used as inner solvers to compute approximate Jacobians. This article analyzes the impact of these approximate Jacobians within the variable projection method and introduces stopping criteria to ensure convergence. We also present numerical experiments where we apply the proposed method to solve a blind deconvolution problem to illustrate and confirm our theoretical results.
\end{abstract}

\begin{keywords} Variable projection, Tikhonov regularization, inverse problems, LSQR
\end{keywords}

\begin{AMS} 65F22, 65F10, 68W40
\end{AMS}

\section{Introduction}

We consider discrete ill-posed inverse problems of the form
\begin{equation}\label{eq: nonlinear}
\bA(\by)\bx  \approx \bb =  \bb_{\rm true} + \bepsilon \quad \mbox{ with } \bA(\by_{\rm true}) \bx_{\rm true} = \bb_{\rm true},
\end{equation}
where the vector $\bb_{\rm true} \in \mathbb{R}^m$ denotes an unknown error-free vector associated with the available data and $\bepsilon \in \mathbb{R}^m$ is an unknown vector that represents the noise/errors in $\bb$. The matrix $\bA(\by) \in \mathbb{R}^{m \times n}$ with $m\geq n$ models a forward operator and is typically severely ill-conditioned. This paper assumes that $\bA$ is unknown but can be parametrized by a vector $\by\in \mathbb{R}^r$ with $r \ll n$. We aim to compute good approximations of $\bx_{\rm true}$ and $\by_{\rm true}$, given a data vector $\bb$ and a matrix function that maps the unknown vector $\by$ to an $m \times n$ matrix $\bA$. Then, we would like to solve
\begin{equation}\label{eq: nsl2}
\min_{\bx, \by}  \mathcal{F}(\bx, \by) = \min_{\bx, \by} \dfrac{1}{2}\|\bA(\by)\bx - \bb\|_2^2+ \dfrac{\lambda^2}{2}\|\bL \bx\|_2^2,
\end{equation}
where $\lambda>0$ is called the \emph{regularization parameter} and $\bL \in \mathbb{R}^{q\times n}$ is a \emph{regularization operator}. We call these problems separable nonlinear inverse problems since the observations depend nonlinearly on the vector of unknown parameters $\by$ and linearly on the desired solution~$\bx$. We will assume that $\bL$ verifies that $$\mathcal{N}(\bA(\by)) \cap \mathcal{N} (\bL) = \{0\}$$
for all feasible values of $\by$, where $\mathcal{N} (\bM)$ denotes the null space of the matrix $\bM$, so that the minimization problem \eqref{eq: nsl2} has a unique solution for $\by$ fixed.

This article focuses on the variable projection (VarPro) method introduced in~\cite{golub1973differentiation} to solve (unregularized) separable nonlinear least squares problems, i.e., problems of the form \begin{equation}\label{eq: min}
    \min_{\bx, \by} \dfrac{1}{2}\left\|\bA(\by)\bx - \bb\right\|_2^2.
\end{equation}
VarPro is an efficient method with the main idea of eliminating the linear variable $\bx$ by solving a linear least squares problem for each nonlinear variable $\by$. Therefore, by writing $\bx = \bx(\by) = \bA(\by)^{\dagger}\bb$, where $\bA(\by)^{\dagger}= (\bA(\by)^\top\bA(\by))^{-1}\bA^\top(\by)$ is the Moore-Penrose pseudoinverse of $\bA(\by)$, the functional to be minimized is reduced to a functional of the variable $\by$ only, leading to the following minimization problem
\begin{equation}\label{eq: minProjec}
    \min_{\by} \dfrac{1}{2}\|\bA(\by)\bA(\by)^\dagger\bb - \bb\|_2^2,
\end{equation}
which is shown in \cite{golub1973differentiation} to have the same solution as \eqref{eq: min}.
This \emph{reduced} minimization problem, which is a nonlinear least squares problem, can be solved using the Gauss-Newton (GN) method.
If, for every $\by$,  we write $\mathcal{P}^{\perp}_{\bA(\by)} = \bI - \bA(\by)\bA(\by)^{\dagger}$ for the orthogonal projector onto the orthogonal complement of the column space of $\bA(\by)$, the reduced problem \eqref{eq: minProjec} can be re-written as $\min_{\by} \frac{1}{2}\| - \mathcal{P}^{\perp}_{\bA(\by)}\bb\|_2^2$.
To solve it using the GN method, an analytic expression of the Jacobian matrix of $\mathcal{P}^{\perp}_{\bA(\by)}\bb$ with respect to the variable $\by$ is given in \cite{golub1973differentiation}.

Several works have shown that separating the linear variable $\bx$ from the nonlinear variable $\by$ as proposed by VarPro speeds up the convergence of iterative methods to solve \eqref{eq: min}.
However, computing the Jacobian might be difficult and expensive. For this reason, some approximations of the Jacobian have been suggested by Kaufman in~\cite{kaufman1975variable} and by Ruano, Jones, and Fleming in~\cite{ruano1991new} (see also~\cite{ruhe1974algorithms}). For a survey on the VarPro method, its variants, and applications, we refer the reader to \cite{golub2003separable}. More recent applications of VarPro can be found in \cite{dong2022numerical, newman2021train, pereyra2019imaging}.

The use of VarPro for solving regularized separable nonlinear least squares problems of the form \eqref{eq: nsl2} was first introduced by Chung and Nagy in~\cite{chung2010efficient}, for the case when $\bL$ is the identity matrix.  The authors presented an efficient method that uses a hybrid Krylov subspace approach to overcome the high computational cost of solving \eqref{eq: nsl2} for large-scale inverse problems and applied it to blind deconvolution problems. In \cite{gazzola2021regularization}, the authors modified the approach in~\cite{chung2010efficient} by incorporating an inexact Krylov method to solve the linear subproblem.

In~\cite{Espanol_2023}, VarPro was extended to solve \eqref{eq: nsl2} for general regularization matrices $\bL$. Its new version was named GenVarPro. That work also included expressions for computing the Jacobian and the approximations given by Kaufman and by Ruano, Jones, and Fleming, using the generalized singular value decomposition and joint spectral decomposition of forward and regularization operators when they are available or feasible to compute. For large-scale inverse problems, projection-based iterative methods and generalized Krylov subspace methods were employed for solving the linear subproblems needed to approximate Jacobians. Numerical examples, particularly in the context of large-scale two-dimensional imaging problems like semi-blind deblurring, demonstrated the effectiveness of GenVarPro.

In this article, we will present a rigorous convergence analysis to understand the impact of incorporating an iterative method in GenVarPro to approximate the solution $\bx$ at each iteration in order to compute an approximate Jacobian matrix when solving the reduced problem by the Gauss-Newton method. Furthermore, we will confirm our theoretical analysis by applying the proposed method to a blind deconvolution problem. The algorithm we introduce, which we call Inexact-GenVarPro, works with a new approximation of the Jacobian where we replace the exact value of $\bx$ by an approximate solution obtained by the LSQR algorithm~\cite{paige1982lsqr}. Thus, the analysis included here does not apply directly to the methods proposed in \cite{kaufman1975variable} or \cite{ruano1991new}, where different approximated Jacobians are used. Our results do not apply to the convergence of the methods in ~\cite{chung2010efficient, Espanol_2023, gazzola2021regularization} either, because in these methods, the regularization parameter is chosen at each iteration, whereas we keep it fixed all along the algorithm.

Similar analyses have been presented in different contexts. For instance, in~\cite{devolder2014first}, the authors introduce the concept of an inexact first-order \emph{oracle} and examine its impact on various first-order methods utilized in smooth convex optimization. The inexact oracle arises naturally in the context of smoothing techniques, Moreau–Yosida regularization, Augmented Lagrangians, and similar situations. In~\cite{van2021variable}, an extension of VarPro for nonsmooth problems of the form
$$\min_{\bx,\by} f(\bx, \by) + r_1(\bx) + r_2(\by),$$
is introduced, which includes an inexact adaptive algorithm that solves projection subproblems through iterative methods. The authors made a convergence analysis of the method
 for the case when $f(\bx, \by ) + r_1(\bx)$ is strongly convex in $\bx$, so that efficient iterative methods can be used for solving $\min_{\bx} f(\bx, \by) + r_1(\bx).$ In comparison, our analysis is less general since it only focuses on problems of the form~\eqref{eq: nonlinear}, but we pose specific conditions in terms of the matrices involved.

Recently, a secant variable projection (SVP) method for solving separable nonlinear least squares problems, which employs rank-one updates to estimate the Jacobian matrices efficiently, was proposed in~\cite{song2020secant}, along with a convergence analysis. Here, we adapt parts of that analysis for our study of the Inexact-GenVarPro convergence rate.

The paper is organized as follows. In Section \ref{sec:varpro}, we reviewed the GenVarPro method and introduced the Inexact-GenVarPro method. In Section \ref{sec: convergence}, we present a convergence analysis of the Inexact-GenVarPro method. Some numerical experiments in Section \ref{sec: numexamples} verify our convergence results, and the conclusions follow in Section \ref{sec:conclusions}.

\section{Variable Projection Methods for Regularized Problems}\label{sec:varpro}

The main idea behind VarPro~\cite{golub1973differentiation} (and GenVarPro~\cite{Espanol_2023}) is to remove the variable $\bx$ from the problem formulation and provide a reduced functional to minimize only with respect to $\by$. That is to say, to solve the problem \eqref{eq: nsl2}, we can apply the GN method to the functional $f(\by) = \mathcal{F}(\bx(\by), \by)$, where $\bx(\by)$ is the solution of the minimization problem
\begin{equation}\label{eq: Fx}
\min_{\bx}\mathcal{F}(\bx, \by) =  \min_{\bx} \dfrac{1}{2} \left\|\left[\begin{array}{c} \bA(\by) \\ \lambda \bL\end{array}\right] \bx-
\left[\begin{array}{c} \bb \\ \bf{0} \end{array}\right]\right\|_2^2.
\end{equation}
Under the assumption $\mathcal{N}(\bA(\by)) \cap \mathcal{N} (\bL) = \{0\}$, the matrix $\left[\begin{array}{c} \bA(\by) \\ \lambda \bL\end{array}\right]$ has rank $n$ and
this problem has a closed-form solution
\begin{equation*}
\bx(\by) = \left[\begin{array}{c} \bA(\by) \\ \lambda \bL\end{array}\right]^\dagger\left[\begin{array}{c} \bb \\ \bf{0} \end{array}\right]
\end{equation*}
that can be used to rewrite the nonlinear problem with respect to the variable $\by$, obtaining the reduced minimization problem
\begin{equation}\label{eq: Psiy}
    \min_{\by}f(\by) = \min_{\by} \frac{1}{2}\|\bff_{\lambda,\bL}(\by)\|_2^2,
\end{equation}
where $\bff_{\lambda,\bL}: \mathbb{R}^r \to \mathbb{R}^{m+q}$ is defined as
\begin{align}\label{eq:main_function}
\bff_{\lambda,\bL}(\by) &=  \left[\begin{array}{c} \bA(\by) \\ \lambda \bL\end{array}\right] \left[\begin{array}{c} \bA(\by) \\ \lambda \bL\end{array}\right]^\dagger\left[\begin{array}{c} \bb \\ \bf{0} \end{array}\right]-
\left[\begin{array}{c} \bb \\ \bf{0} \end{array}\right] \nonumber \\
&= \left(\left[\begin{array}{c} \bA(\by) \\ \lambda \bL\end{array}\right] \left[\begin{array}{c} \bA(\by) \\ \lambda \bL\end{array}\right]^\dagger - \bI
\right) \left[\begin{array}{c} \bb \\ \bf{0} \end{array}\right].
\end{align}

To simplify notation, we define
\[ \bA_{\lambda, \bL}(\by)=\left[\begin{array}{c} \bA(\by) \\ \lambda \bL\end{array}\right]  \ \mbox{ and } \ \mathcal{P}^\perp_{\bA_{\lambda,\bL}}(\by)= \bI-\bA_{\lambda,\bL}(\by)\bA_{\lambda, \bL }^\dagger(\by),\]
and write only $\bA_{\lambda, \bL}$ and $\mathcal{P}^\perp_{\bA_{\lambda,\bL}}$ instead of $\bA_{\lambda,\bL}(\by)$ and $\mathcal{P}^\perp_{\bA_{\lambda,\bL}}(\by)$ for even more simplification.

\subsection{GenVarPro}

To solve \eqref{eq: Psiy}, we can apply the Gauss-Newton method, whose iterations are defined by
$$
\by^{(k+1)} = \by^{(k)} + \bs^{(k)}, \,k = 0,1,2,...,
$$
where $\bs^{(k)}$ is defined as
\begin{equation*}
\bs^{(k)} = \arg\min_{\bs}\left\| \bJ_{\bff_{\lambda, \bL}}(\by^{(k)})\bs + \bff_{\lambda, \bL}(\by^{(k)}) \right\|^2_2
\end{equation*}
with $\bJ_{\bff_{\lambda, \bL}}\colon \mathbb{R}^{r}\to\mathbb{R}^{(m+q)\times r}$ being the Jacobian matrix of $\bff_{\lambda, \bL}$. Implementations of GN usually include a step size $\alpha^{(k)}$ so that
$\by^{(k+1)} = \by^{(k)} + \alpha^{(k)} \bs^{(k)}$; here, $\alpha^{(k)}=1$.

The $j$-th column of  $\bJ_{\bff_{\lambda, \bL}}$ can be computed by
\begin{align*}
\frac{\partial \bff_{\lambda, \bL}(\by)}{\partial \by_j} & = \frac{\partial }{\partial \by_j}\left(\bA_{\lambda, \bL}\bA_{\lambda, \bL}^\dagger \left[\begin{array}{c} \bb \\ \bf{0} \end{array}\right]\right)
 = \left(\frac{\partial \bA_{\lambda, \bL} }{\partial \by_j}\bA_{\lambda, \bL}^\dagger+ \bA_{\lambda, \bL}\frac{\partial \bA_{\lambda, \bL}^\dagger }{\partial \by_j} \right) \left[\begin{array}{c} \bb \\ \bf{0} \end{array}\right].
\end{align*}
By writing $\bA_{\lambda, \bL}^\dagger=(\bA_{\lambda, \bL}^\top\bA_{\lambda, \bL})^{-1}\bA_{\lambda, \bL}^\top $, applying the product rule, and using that for an invertible matrix $\bM(\by)$ the identity $\frac{\partial \bM(\by)^{-1}}{\partial \by_j} = - \bM(\by)^{-1} \frac{\partial \bM(\by)}{\partial \by_j}\,\bM(\by)^{-1} $ holds,
we have that
\begin{align}
    \frac{\partial \bA_{\lambda, \bL}^\dagger}{\partial \by_j} & = (\bA_{\lambda, \bL}^\top\bA_{\lambda, \bL})^{-1} \frac{\partial \bA_{\lambda, \bL}^\top}{\partial \by_j}\mathcal{P}^\perp_{\bA_{\lambda, \bL}}-\bA_{\lambda, \bL}^\dagger \frac{\partial \bA_{\lambda, \bL}}{\partial \by_j} \bA_{\lambda, \bL}^\dagger. \nonumber \\
    & = (\bA_{\lambda, \bL}^\top\bA_{\lambda, \bL})^{-1} \left[ \frac{\partial \bA^\top}{\partial \by_j}\  \bf{0} \right]\mathcal{P}^\perp_{\bA_{\lambda, \bL}} \nonumber - \bA_{\lambda, \bL}^\dagger \left[\begin{array}{c} \frac{\partial \bA}{\partial \by_j} \\ \bf{0} \end{array} \right] \bA_{\lambda, \bL}^\dagger.
\end{align}
Therefore, the $j$-th column of the Jacobian $\bJ_{\bff_{\lambda, \bL}}$ is given by
\begin{align}\label{eq: jthColumnJacobian}
\frac{\partial \bff_{\lambda, \bL}(\by)}{\partial \by_j} & =
\left(\frac{\partial \bA_{\lambda, \bL} }{\partial \by_j}\bA_{\lambda, \bL}^\dagger  \right. {} \nonumber \\
& \left. \quad {}+ \bA_{\lambda, \bL}\left((\bA_{\lambda, \bL}^\top\bA_{\lambda, \bL})^{-1} \left[ \frac{\partial \bA^\top}{\partial \by_j}\ \bf{0} \right]\mathcal{P}^\perp_{\bA_{\lambda, \bL}} \nonumber - \bA_{\lambda, \bL}^\dagger \left[ \begin{array}{c} \frac{\partial \bA}{\partial \by_j} \\ \bf{0} \end{array} \right] \bA_{\lambda, \bL}^\dagger \right) \right) \left[\begin{array}{c} \bb \\ \bf{0} \end{array}\right]\nonumber \\
& = \mathcal{P}^\perp_{\bA_{\lambda, \bL}} \left[ \begin{array}{c} \frac{\partial \bA}{\partial \by_j} \\ \bf{0} \end{array} \right] \bA_{\lambda, \bL}^\dagger \left[\begin{array}{c} \bb \\ \bf{0} \end{array}\right]
+
\left( \mathcal{P}^\perp_{\bA_{\lambda, \bL}}\left[ \begin{array}{c} \frac{\partial \bA}{\partial \by_j} \\ \bf{0} \end{array} \right] \bA_{\lambda, \bL}^\dagger \right)^\top\left[\begin{array}{c} \bb \\ \bf{0} \end{array}\right]. \nonumber \\
\end{align}

\cref{Alg:GenVarPro} summarizes the steps of the GenVarPro method. Notice that the linear problem in step 6 is small because of our assumption on the dimension $r$ of $\by$. Therefore, it can be solved exactly by any direct method. For a discussion on the stopping criteria, see \cite[Section 6]{Espanol_2023}.

\begin{algorithm}[ht]
\caption{GenVarPro Algorithm}
\label{Alg:GenVarPro}

\begin{algorithmic}[1]
\STATE {\bf Input:} A map $\by \mapsto \bA(\by)$, $\bb$, and $\by^{(0)}$
\FOR{$k = 0, 1, \dots$ until a stopping criterion is satisfied}
    \STATE $\bx^{(k)} = \left(\bA(\by^{(k)})^\top\bA(\by^{(k)})+\lambda^2 \bL^\top \bL\right)^{-1}\bA(\by^{(k)})^\top\bb$
    \STATE $\bff_{\lambda, \bL}^{(k)} = \left[\begin{array}{c} \bA(\by^{(k)}) \\ \lambda \bL\end{array}\right] \bx^{(k)}-\left[\begin{array}{c} \bb \\ \mathbf{0} \end{array}\right]$\;
    \STATE Compute the Jacobian matrix $\bJ_{\bff_{\lambda, \bL}}^{(k)}=\bJ_{\bff_{\lambda, \bL}}(\by^{(k)})$\;
    \STATE $\bs^{(k)} = \arg\min_{\bs}\| \bJ_{\bff_{\lambda, \bL}}^{(k)}\bs + \bff_{\lambda, \bL}^{(k)}\|^2_2$\;
    \STATE $\by^{(k+1)} = \by^{(k)} + \bs^{(k)}$\;
\ENDFOR
\end{algorithmic}
\end{algorithm}

\subsection{Inexact-GenVarPro}

Recalling that $\bx(\by) = \bA_{\lambda, \bL}^\dagger(\by) \left[\begin{array}{c} \bb \\ \bf{0} \end{array}\right]$, we can re-write the columns of the Jacobian given by \eqref{eq: jthColumnJacobian}  as
$$[\bJ_{\bff_{\lambda, \bL}}(\by)]_j = \mathcal{P}^\perp_{\bA_{\lambda, \bL}}(\by) \left[ \begin{array}{c} \frac{\partial \bA}{\partial \by_j}(\by) \\ \bf{0} \end{array} \right] \bx(\by) + (\bA_{\lambda, \bL}^\dagger(\by))^\top \left(\frac{\partial \bA}{\partial \by_j}(\by)\right)^\top (\bb- \bA\, \bx(\by)).$$
Note that, even though we have re-formulated the original minimization problem \eqref{eq: nsl2} only in terms of $\by$, the value of $\bx= \bx(\by)$ still appears in the Jacobian needed to solve the reduced minimization problem \eqref{eq: Psiy} using GN. In addition, the residual $$\bff_{\lambda, \bL} (\by) = \bA_{\lambda, \bL}(\by) \bx(\by) - \left[\begin{array}{c} \bb \\ \mathbf{0} \end{array}\right]$$ also depends on $\bx(\by)$. For large-scale problems, computing these values is a computationally expensive task. This motivates the search for alternative, more efficient computational strategies.

A possible way of reducing the computational cost of the algorithm consists in approximating the exact solution $\bx(\by)$ of \eqref{eq: Fx} by applying an iterative method to compute it. Here, we adopt this approach by incorporating the LSQR iterative method.
More precisely, we propose to approximate the Jacobian $\bJ_{\bff_{\lambda, \bL}}(\by^{(k)})$ at the $k$-th iteration of the algorithm by the matrix $\bar\bJ^{(k)}$ whose columns are defined by
\begin{equation}\label{eq: Jappkj}
 [\bar\bJ^{(k)}]_j = \mathcal{P}^\perp_{\bA_{\lambda, \bL}}(\by^{(k)}) \left[ \begin{array}{c} \frac{\partial \bA}{\partial \by_j}(\by^{(k)}) \\ \bf{0} \end{array} \right] \bar{\bx}^{(k)} + (\bA_{\lambda, \bL}^\dagger(\by^{(k)}))^\top \!\!\left(\frac{\partial \bA}{\partial \by_j}(\by^{(k)})\right)^\top \!\! (\bb- \bA\, \bar{\bx}^{(k)}),
\end{equation}
for $j=1,\dots, r$, where $\bar{\bx}^{(k)}$ is the approximate solution of the linear subproblem
\[\min_{\bx} \frac{1}{2}\|\bA(\by^{(k)})\bx - \bb\|_2^2+ \frac{\lambda^2}{2}\|\bL\bx\|_2^2.\]
The iterative algorithm LSQR we apply to compute $\bar{\bx}^{(k)}$ works with a stopping criterion depending on a tolerance $\varepsilon^{(k)} >0$, which we compute at each iteration from an initial tolerance $\varepsilon^{(0)}$.

We call the new algorithm Inexact-GenVarPro.
\cref{Alg:I-GenVarPro} summarizes the \linebreak Inexact-GenVarPro method applied to problem \eqref{eq: minProjec}.

\begin{algorithm}[ht]
\caption{Inexact-GenVarPro Algorithm}
\label{Alg:I-GenVarPro}
\begin{algorithmic}[1]
\STATE {\bf Input:} A map $\by \mapsto \bA(\by)$, $\bb$, $\by^{(0)}$, and $\varepsilon^{(0)}>0$
\FOR{$k = 0, 1, \dots$ until a stopping criterion is satisfied}
    \STATE Compute $\bar \bx^{(k)}$ applying the LSQR algorithm with stopping criterion \eqref{eq: stopping_criterion} for $\varepsilon=\varepsilon^{(k)}$ to the problem
    $\min_{\bx} \frac{1}{2}\|\bA(\by^{(k)})\bx - \bb\|_2^2+ \frac{\lambda^2}{2}\|\bL\bx\|_2^2$
    \STATE $\bg^{(k)} = \left[\begin{array}{c} \bA(\by^{(k)}) \\ \lambda \bL\end{array}\right] \bar \bx^{(k)}-\left[\begin{array}{c} \bb \\ \mathbf{0} \end{array}\right]$\;
    \STATE Compute the approximate Jacobian matrix $\bar\bJ^{(k)}$ according to \eqref{eq: Jappkj}\;
    \STATE $\bt^{(k)} = \arg\min_{\bs}\| \bar\bJ^{(k)}\bs + \bg^{(k)}\|^2_2$\;
    \STATE $\by^{(k+1)} = \by^{(k)} + \bt^{(k)}$\;
    \STATE $\varepsilon^{(k+1)} = \varepsilon^{(k)}/2$\;
\ENDFOR
\end{algorithmic}

\end{algorithm}

In the next section, we present a convergence analysis of Inexact-GenVarPro and provide conditions on the tolerance $\varepsilon^{(0)}$ for the stopping criterion of LSQR to ensure convergence.
As we will see, in step 3 of Algorithm \ref{Alg:I-GenVarPro}, the LSQR method could be replaced by any iterative method provided that the same stopping criterion is used.

\section{Convergence Analysis}\label{sec: convergence} In this section,
we analyze the convergence of \linebreak Inexact-GenVarPro. Specifically, we will first prove bounds for the required accuracy of the approximations $\bar{\bx}^{(k)}$ to the solutions $\bx(\by^{(k)})$ and then use them to deduce bounds for the algorithm's convergence rate.

Suppose that, at iteration $k$ of the algorithm, the matrix $\bar \bJ^{(k)}$ is an approximation of the Jacobian matrix $\bJ^{(k)}=\bJ_{\bff_{\lambda, \bL}}(\by^{(k)})$, and $\bg^{(k)}$ is an approximation of $\bff_{\lambda, \bL}^{(k)} =\bff_{\lambda, \bL}(\by^{(k)})$. Then, the Inexact-GenVarPro iteration is defined by
$$\by^{(k+1)}=\by^{(k)} + \bt^{(k)} \mbox{ with } \bt^{(k)}= -\left((\bar \bJ^{(k)})^\top \bar \bJ^{(k)}\right)^{-1} (\bar \bJ^{(k)})^\top\bg^{(k)}.$$
 Assuming that $\by^\ast$ is a minimizer of $\mathnormal{f}(\by)=\frac{1}{2}\|\bff_{\lambda,\bL}(\by)\|_2^2,$
 our aim is to bound the errors $\| \be^{(k)}\|_2$, where  $\be^{(k)} = \by^{(k)} - \by^\ast$, for $k=0,1,2,\dots$.

Taking into account that $\nabla \mathnormal{f}(\by^\ast)=  \bJ_{\bff_{\lambda, \bL}}^\top(\by^\ast)\bff_{\lambda, \bL}(\by^\ast)=0$,  it follows that
\begin{align}\label{eq: BkBkek}
    (\bar\bJ^{(k)})^\top \bar\bJ^{(k)} \be^{(k+1)} & = \left[ \nabla \mathnormal{f}(\by^{(k)}) - (\bar\bJ^{(k)})^\top \bg^{(k)} \right] + \left[(\bar\bJ^{(k)})^\top \bar \bJ^{(k)} - \nabla^2 \mathnormal{f} (\by^\ast) \right]\be^{(k)}  \nonumber \\
    & \qquad \qquad +   \left[ -\nabla \mathnormal{f}(\by^{(k)}) + \nabla \mathnormal{f} (\by^\ast) + \nabla^2 \mathnormal{f} (\by^\ast) \be^{(k)} \right].
\end{align}
If $\bJ_{\bff_{\lambda, \bL}}(\by)$ is Lipschitz continuous, we have that
$$\left\| -\nabla \mathnormal{f}(\by^{(k)}) + \nabla \mathnormal{f} (\by^\ast) + \nabla^2 \mathnormal{f} (\by^\ast) \be^{(k)} \right\|_2 = \mathcal{O}(\|\be^{(k)} \|_2^2).$$
Then, the first two terms in \eqref{eq: BkBkek} will determine the convergence rate. For both, we need to bound the errors between $\bar\bJ^{(k)}$ and $\bJ^{(k)}$, and  $\bg^{(k)}$ and $\bff_{\lambda, \bL}^{(k)}$. To do so, we first give some upper bounds for the approximation of $\bx(\by^{(k)})$ and its residual, continue with bounds for our approximations of the Jacobians, and finally, state and prove our main result.

\subsection{Inner solver bounds} A key invariant appearing in our bounds is the condition number of the matrices involved. In the sequel,  $\kappa_2(\bM) = \|\bM\|_2\|\bM^\dagger\|_2$ denotes the condition number of a matrix $\bM$.

In the next lemma, based on \cite[Theorem 20.1]{higham2002accuracy}, we will prove bounds that we will use to estimate the error of our approximate computations. We state the specific formulation we need and include its complete proof for the reader's convenience.

\begin{lemma}\label{lem:approx_solution_M}
Let $\bM \in \mathbb{R}^{l\times n}$, with $l\ge n $, be a matrix of full rank, $\bd \in \mathbb{R}^{l\times 1}$, and $\bx\in \mathbb{R}^{n\times 1}$ the solution of the problem
\begin{equation*}\label{eq: Mx-b}
    \min_\bx \left\| \bM \bx - \bd \right\|^2_2.
\end{equation*}
For $\varepsilon >0$ such that $\kappa_2(\bM) \varepsilon<1$, let $\bar \bx$ be an approximation of $\bx$ computed using an iterative method with stopping criterion
\begin{equation}\label{eq: stopping_criterion_M}
\frac{\|\bM^\top \br^{(i)}\|_2}{\|\br^{(i)}\|_2\|\bM\|_2}<\varepsilon,
\end{equation}
where $\br^{(i)} = \bd -\bM \bx^{(i)}$ with $\bx^{(i)}$ the approximate solution given by the $i$-th iteration of the algorithm.
Then, if $\br=\bd -\bM \bx$ and $\bar \br = \bd -\bM \bar\bx$ are the corresponding residuals, we have:
$$\|\bx-\bar \bx\|_2 <  \frac{ 2\kappa_2^2(\bM)}{1-\varepsilon\, \kappa_2(\bM)} \frac{\|\bd\|_2}{\|\bM\|_2} \varepsilon \quad \hbox{ and } \quad
\|\br - \bar \br\|_2 <   \frac{ 2 \kappa_2(\bM)}{1-\varepsilon\, \kappa_2(\bM)}\|\bd\|_2\varepsilon.$$
\end{lemma}

\begin{proof}
Throughout the proof, $\| \cdot \|$ will denote the $2$-norm for vectors and matrices.
Following \cite{paige1982lsqr}, we can check that, if
$$ \bE = - \frac{\bar\br \bar\br^\top \bM}{\|\bar\br\|^2},$$
since $\left( \bI -\frac{\bar\br \bar\br^\top}{\|\bar\br\|^2} \right) \bar\br =0$, we have that
\begin{align*}
(\bM+\bE)^\top ( \bd - (\bM+\bE)\bar \bx) &=  (\bM^\top +\bE^\top) (\bar \br - \bE \bar\bx) \\
&=
\bM^\top \left( \bI -\frac{\bar\br \bar\br^\top}{\|\bar\br\|^2} \right) \left(\bar\br +\frac{\bar\br \bar\br^\top \bM\bar\bx }{\|\bar\br\|^2}\right) =0.
\end{align*}
Then, the vector $\bar\bx$ is a solution of the perturbed problem
$$\min_{\bx}\left\|(\bM + \bE)\bx- \bd\right\|^2.$$

The bound for $\| \bar\bx - \bx\|$ will be obtained by modifying \cite[Theorem 20.1]{higham2002accuracy} for the case when $\Delta \bb=0$.
Note that
\begin{equation}\label{eq: normE}
\|\bE\| = \left\| \frac{\bar\br \bar\br^\top \bM}{\|\bar\br\|^2} \right\| = \frac{\|\bM^\top \bar\br \bar\br^\top\|}{\|\bar\br\|^2} \leq \frac{\|\bM^\top \bar\br\|}{\|\bar\br\|} < \varepsilon \|\bM\|,
\end{equation}
where the last inequality is a consequence of the stopping criterion \eqref{eq: stopping_criterion_M}.

For simplicity, let $\bB = \bM+\bE$. We have:
\begin{equation}\label{eq:x_diff}
\bar\bx - \bx = \bB^\dagger \bd - \bx = \bB^\dagger \left(\br + (\bB - \bE) \bx\right) - \bx = \bB^\dagger\br - \bB^\dagger \bE \bx.
\end{equation}

Now, $\bB^\dagger=\bB^\dagger\bB \bB^\dagger= \bB^\dagger \mathcal{P}_{\bB}$ and, since $\br = \mathcal{P}_{\bM}^\perp \bd$, then $\mathcal{P}_{\bM}^\perp\br = (\mathcal{P}_{\bM}^\perp)^2 \bd =\mathcal{P}_{\bM}^\perp \bd = \br$. This implies that
$$\|\bB^\dagger\br\|= \|\bB^\dagger \mathcal{P}_{\bB}\mathcal{P}_{\bM}^\perp \br\| \le \|\bB^\dagger\| \|\mathcal{P}_{\bB}\mathcal{P}_{\bM}^\perp\| \| \br\|.$$

By \cite[Lemma 20.11]{higham2002accuracy}, under the assumption that $\operatorname{rank}(\bB) = \operatorname{rank}(\bM)$, since $\|\bM^\dagger\| \|\bE\|\le \|\bM^\dagger\| \|\bM\| \varepsilon = \kappa_2(\bM) \varepsilon <1$ due to Inequality \eqref{eq: normE} and our assumption on $\varepsilon$, it follows that
$$\|\bB^\dagger\| \le \frac{\|\bM^\dagger \|}{1 -\|\bM^\dagger\| \|\bE\| }.$$
In addition, by \cite[Theorem 2.3]{stewart1977perturbation}, $\|\mathcal{P}_{\bB}\mathcal{P}_{\bM}^\perp\| = \|\mathcal{P}_{\bM}\mathcal{P}_{\bB}^\perp\|$ and, taking into account that
$$\mathcal{P}_{\bM}\mathcal{P}_{\bB}^\perp = \mathcal{P}_{\bB}^\perp \mathcal{P}_{\bM} = \mathcal{P}_{\bB}^\perp \bM \bM^\dagger = \mathcal{P}_{\bB}^\perp (\bB - \bE) \bM^\dagger =-\mathcal{P}_{\bB}^\perp  \bE \bM^\dagger, $$
we deduce that
\begin{equation}\label{eq:projections_norm}
\|\mathcal{P}_{\bB}\mathcal{P}_{\bM}^\perp\| \le \| \bE \| \|\bM^\dagger \|.
\end{equation}
Therefore,
$$\|\bB^\dagger\br\| \le \frac{\|\bM^\dagger \|^2 \| \bE \|}{1 -\|\bM^\dagger\| \|\bE\| } \| \br\|.$$
Using the expression from \eqref{eq:x_diff}, the previous inequality, together with
$$\|\bB^\dagger \bE \bx \| \le \| \bB^\dagger\| \|\bE\| \|\bx \| \le \frac{\|\bM^\dagger \| \| \bE \|}{1 -\|\bM^\dagger\| \|\bE\| } \|\bx \|, $$
 implies that
\begin{equation}\label{eq:norm_x_diff}
\| \bar \bx - \bx\| \le \frac{\|\bM^\dagger \| \| \bE \|}{1 -\|\bM^\dagger\| \|\bE\|} (\| \bM^\dagger\| \| \br \|+\| \bx\|).
\end{equation}
Since $ \|\br \|= \|\mathcal{P}_{\bM}^\perp  \bd\| \le \|\bd\|$,   $\|\bx\|= \|\bM^\dagger \bd\| \le \|\bM^\dagger\| \|\bd\|$, and  $\|\bE \| <\|\bM\| \varepsilon$, it follows that
$$\| \bar \bx - \bx\| \le \frac{\|\bM^\dagger \| \| \bE \|}{1 -\|\bM^\dagger\| \|\bE\| }  2\| \bM^\dagger\| \| \bd \|<\frac{2\kappa_2^2(\bM)}{1- \kappa_2(\bM)\varepsilon} \frac{\|\bd\|}{\|\bM\|} \varepsilon.$$

To obtain an upper bound for $\|\bar\br-\br\|$, we first use Identity \eqref{eq:x_diff} to re-write:
\begin{align*}
\bar\br - \br & =\bM (\bx - \bar\bx) = \bB  (\bx - \bar\bx)- \bE(\bx - \bar\bx) \\ &= \bB \bB^\dagger (\bE \bx - \br)  - \bE(\bx - \bar\bx) = - \mathcal{P}_\bB  \br + \mathcal{P}_\bB \bE \bx + \bE (\bar\bx - \bx).
\end{align*}
By Inequality \eqref{eq:projections_norm},
$$  \| \mathcal{P}_\bB  \br \| = \| \mathcal{P}_{\bB}\mathcal{P}_{\bM}^\perp \br\| \le \| \mathcal{P}_{\bB}\mathcal{P}_{\bM}^\perp\| \| \br\|\le \| \bE\| \| \bM^\dagger\| \| \br\|$$
and, using Inequality \eqref{eq:norm_x_diff},
\begin{align*}
\|\bar\br - \br \| & \le  \|\mathcal{P}_\bB  \br \| + \| \mathcal{P}_\bB \bE \bx\| + \|\bE (\bar\bx - \bx) \| \le
\| \bE\| \| \bM^\dagger\| \| \br\| + \|\bE\| \|\bx\| + \| \bE\| \|\bar\bx - \bx\| \\
&=   \|\bE\| \left(\| \bM^\dagger\| \| \br\| + \|\bx\|\right) + \frac{ \| \bM^\dagger\| \| \bE\|^2}{1 -\|\bM^\dagger\| \|\bE\|}  \left( \| \bM^\dagger\| \| \br\| + \|\bx\|\right) \\
&= \frac{ \| \bM^\dagger\| \| \bE\|}{1 -\|\bM^\dagger\| \|\bE\|}  \left(\| \br\| + \frac{\|\bx\|}{\| \bM^\dagger\|}\right).
\end{align*}
Finally, taking into account that $\|\br\|\le \| \bd\|$,  $\|\bx\|\le \|\bM^\dagger\| \| \bd\|$, and $\| \bM^\dagger\| \| \bE\|< \kappa_2(\bM)\varepsilon$, we obtain
$$\|\bar\br - \br \| < \frac{2\kappa_2(\bM)}{1- \kappa_2(\bM)\varepsilon} \|\bd\|\varepsilon,$$
which is the stated bound.
\end{proof}

In the Inexact-GenVarPro method described by \cref{Alg:I-GenVarPro}, in Step 3, we apply an iterative method to achieve the following task: for a fixed $\by$, obtain an approximate solution of \eqref{eq: Fx}. To do this, we use the LSQR iterative algorithm with stopping criterion
\begin{equation}\label{eq: stopping_criterion}
\frac{\|\bA_{\lambda, \bL}(\by)^\top \br^{(i)}\|_2}{\|\br^{(i)}\|_2\|\bA_{\lambda, \bL}(\by)\|_2}<\varepsilon,
\end{equation}
for a prescribed tolerance $\varepsilon>0$ (here, $\br^{(i)}$ is the residual at the $i$-th iteration).
The previous lemma applied to $\bM =\bA_{\lambda, \bL}(\by)$ and $ \bd = \left[\begin{array}{c} \bb \\ \bf{0} \end{array}\right]$ enables us to estimate the approximation error: if $\bx$ is the (exact) solution of \eqref{eq: Fx}, $\bar \bx$ the approximate solution computed by the LSQR algorithm with stopping criterion \eqref{eq: stopping_criterion} for a sufficiently small tolerance $\varepsilon>0$,
and  $\br_{\lambda, \bL}=\left[\begin{array}{c} \bb \\ \bf{0} \end{array}\right] -\bA_{\lambda, \bL}(\by)\bx$ and $ \bar \br_{\lambda, \bL} = \left[\begin{array}{c} \bb \\ \bf{0} \end{array}\right] -\bA_{\lambda, \bL}(\by)\bar\bx$ are the corresponding residuals, then
\begin{equation}\label{eq:approx_solution}
\|\bx-\bar \bx\|_2 <  \frac{ 2\kappa_2^2(\bA_{\lambda, \bL}(\by))}{1-\varepsilon\, \kappa_2(\bA_{\lambda, \bL}(\by))} \frac{\|\bb\|_2}{\|\bA_{\lambda, \bL}(\by)\|_2} \varepsilon
\end{equation}
and
\begin{equation}\label{eq:approx_residual}
\|\br_{\lambda, \bL} - \bar \br_{\lambda, \bL}\|_2 <   \frac{ 2 \kappa_2(\bA_{\lambda, \bL}(\by))}{1-\varepsilon\, \kappa_2(\bA_{\lambda, \bL}(\by))}\|\bb\|_2\varepsilon.
\end{equation}

\subsection{Main result}
This subsection is devoted to proving our main theoretical result. We will show that if
$\by^\ast$ is a minimizer of $\mathnormal{f}(\by)=\frac{1}{2}\|\bff_{\lambda,\bL}(\by)\|_2^2$,
and $\by^{(0)}$ is sufficiently close to $\by^*$, we can choose an initial tolerance $\varepsilon^{(0)}$ for the stopping criterion of LSQR so that the sequence $\by^{(k)}$, $k=1,2,...$, computed by Algorithm \ref{Alg:I-GenVarPro} (Inexact-GenVarPro) converges to $\by^*$. To do so, we will bound the errors $\| \by^{(k)} - \by^*\|_2$.

We start by applying the bounds proved in the previous subsection to analyze how the Jacobian of the function $\bff_{\lambda, \bL}(\by)$ defined in \eqref{eq:main_function} changes when the exact solution $\bx$ of \eqref{eq: Fx} is replaced with an LSQR-approximation $\bar \bx$. This will enable us to bound $\|\bJ_{\bff_{\lambda, \bL}}(\by^{(k)}) - \bar\bJ^{(k)}\|_2$,  where  $\bar\bJ^{(k)}$ is the approximate Jacobian defined in \eqref{eq: Jappkj}.

\begin{lemma}\label{lem:approx_Jacobian} For a fixed $\by$, let $\bJ= \bJ_{\bff_{\lambda, \bL}}(\by)$ be the Jacobian of $\bff_{\lambda, \bL}(\by)$,
$$\bJ_j = \mathcal{P}^\perp_{\bA_{\lambda, \bL}} \left[ \begin{array}{c} \frac{\partial \bA}{\partial \by_j} \\ \bf{0} \end{array} \right] \bx(\by) +\left(\frac{\partial \bA}{\partial \by_j} \bA_{\lambda, \bL}^\dagger\right)^\top   \br(\by) \qquad \hbox{ for } j=1,\dots, r,$$
with $\br(\by) = \bb - \bA \bx(\by)$, and $\bar \bJ$ the Jacobian where we replace $\bx(\by)$ by an approximate solution $\bar \bx$ satisfying \eqref{eq: stopping_criterion} for a sufficiently small tolerance $\varepsilon>0$ and $\br(\by)$ by $\bar \br= \bb - \bA \bar \bx$.
Then, we have that
$$\|\bar \bJ - \bJ\|_2
 < 4 \sqrt{r(m+q)} \max_j \left\|\frac{\partial \bA}{\partial \by_j}\right\|_2 \frac{ \kappa_2^2(\bA_{\lambda, \bL})}{1-\varepsilon\, \kappa_2(\bA_{\lambda, \bL})} \frac{\|\bb\|_2}{\|\bA_{\lambda, \bL}\|_2} \varepsilon.$$
\end{lemma}
\begin{proof} First, note that
$$\|\bar \bJ - \bJ\|_2 \leq \sqrt{r} \|\bar \bJ - \bJ\|_1 = \sqrt{r} \max_j \left\{\| \bar \bJ_j - \bJ_j\|_1\right\} \leq \sqrt{r (m+q)} \max_j \left \{ \| \bar \bJ_j - \bJ_j\|_2\right\}.$$
Now, for a fixed $j$, we have
\begin{align*}
 \| \bar \bJ_j - \bJ_j\|_2  &= \Big\|\mathcal{P}^\perp_{\bA_{\lambda, \bL}} \left[ \begin{array}{c} \frac{\partial \bA}{\partial \by_j} \\ \bf{0} \end{array} \right] (\bar \bx - \bx) + \left(\frac{\partial \bA}{\partial \by_j} \bA_{\lambda, \bL}^\dagger\right)^\top  (\bar \br - \br) \Big\|_2
\\
& \le \left\|\frac{\partial \bA}{\partial \by_j}\right\|_2 \|\bar\bx- \bx\|_2 + \left\|\frac{\partial \bA}{\partial \by_j}\right\|_2 \| \bA_{\lambda, \bL}^\dagger\|_2 \|\bar\br- \br\|_2.
\end{align*}
The result follows taking into account the bound for $\|\bar\bx- \bx\|_2$ from \eqref{eq:approx_solution}, the fact that $\|\bar\br - \br\|_2\le \|\bar\br_{\lambda, \bL} - \br_{\lambda, \bL}\|_2$, and the bound from \eqref{eq:approx_residual}.
\end{proof}

We are now ready to state and prove our main result. We keep our previous notation.

\begin{theorem} \label{thm: main} Let $\by^*$ be a solution of the problem \eqref{eq: Psiy} and $\bJ^*=\bJ(\by^*)$, where $\bJ= \bJ_{\bff_{\lambda, \bL}}(\by)$ is the Jacobian of $\bff_{\lambda, \bL}(\by)$.
Assume that $\bJ$ satisfies the Lipschitz condition $\|\bJ(\by) - \bJ^\ast \|_2\le L \| \by - \by^\ast\|_2$ with a constant $L>0$ for all $\by$ in a neighborhood of $\by^\ast$, and that there are constants $M>0$,  $\kappa>0$, and $\gamma >0$ such that $\max_j \left\|\frac{\partial \bA (\by)}{\partial \by_j}\right\|_2 \le M$, $\kappa_2(\bA_{\lambda, \bL}(\by))\le \kappa$, and $\| \bA_{\lambda, \bL}(\by)\|_2 \ge \gamma$ for all $\by$ in a neighbourhood of $\by^*$.
Then, if $\|\by^{(0)}- \by^*\|_2$ is sufficiently small, one can choose adequate tolerances $\varepsilon^{(k)}>0$ so that the sequence $(\by^{(k)})_{k\ge 0}$ generated by the  Inexact-GenVarPro Algorithm satisfies $\|\by^{(k)}-\by^*\|_2\le \dfrac{1}{2^k}$ for every $k\in \mathbb{N}$.
\end{theorem}

\begin{proof}
Throughout the proof, we will work with $\| \cdot \|_2$ for vectors and matrices. To simplify notation, we will not write the subscript $2$.

For $k\ge 0$, let $\bar \bJ^{(k)}$ be the approximate Jacobian at $\by^{(k)}$ obtained as in \eqref{eq: Jappkj}; that is, we replace the solution $\bx^{(k)}=\bx(\by^{(k)})$ of $\min_{\bx} \frac{1}{2} \| \bA_{\lambda, \bL}(\by^{(k)}) \bx - [ \bb^\top \ \mathbf{0} ]^\top\|^2$ with an approximate solution $\bar \bx^{(k)}$ satisfying \eqref{eq: stopping_criterion} for our chosen tolerance $\varepsilon^{(k)}>0$. We also write
$$\bff^{(k)} = \bff_{\lambda,\bL}(\by^{(k)}) = \bA_{\lambda, \bL}(\by^{(k)}) \bx^{(k)} - \left[\begin{array}{c} \bb \\ \bf{0} \end{array}\right]
\quad \mbox{and} \quad
\bg^{(k)} = \bA_{\lambda, \bL}(\by^{(k)}) \bar \bx^{(k)} - \left[\begin{array}{c} \bb \\ \bf{0} \end{array}\right].$$

Assume $\alpha>0$ is an upper bound for $\|\bJ^\ast\|$ and $\| \bar\bJ^{(k)}\|$ for all $k$.
Set $\beta= \|((\bJ^\ast)^\top \bJ^\ast)^{-1}\|$ and $\delta^\ast= \|  (\bJ^\ast)^\top \bJ^\ast- \nabla^2 \mathnormal{f} (\by^\ast) \| $. We assume that $\delta^*$ is small, that is to say, that $(\bJ^\ast)^\top \bJ^\ast$ is a close approximation to $\nabla^2 \mathnormal{f} (\by^\ast)$, which is a usual assumption when applying the Gauss-Newton method (see for instance \cite[Section 10.3]{nocedal2006numerical}).

For every $k\ge 0$, let $\be^{(k)}= \by^{(k)} - \by^*$.
From Identity \eqref{eq: BkBkek}, we have that
\begin{align}\label{eq: ek+1}
     \be^{(k+1)} & = ((\bar\bJ^{(k)})^\top \bar\bJ^{(k)})^{-1} \Big(
     \left[ -\nabla \mathnormal{f}(\by^{(k)}) + \nabla \mathnormal{f} (\by^\ast) + \nabla^2 \mathnormal{f} (\by^\ast) \be^{(k)} \right] +  \nonumber\\
    &  + \left[ \nabla \mathnormal{f}(\by^{(k)}) - (\bar\bJ^{(k)})^\top \bg^{(k)} \right]
     + \left[(\bar\bJ^{(k)})^\top \bar\bJ^{(k)} - \nabla^2 \mathnormal{f} (\by^\ast) \right]\be^{(k)}
        \Big).
\end{align}
We will now bound the norm of each of the three terms in the above expression in terms of $\varepsilon^{(k)}$ and $\| \be^{(k)}\|$.

By Taylor expansion and the smoothness of $f$, there is a constant $T>0$ such that
\begin{equation}\label{eq:taylor}
\|-\nabla \mathnormal{f}(\by^{(k)}) + \nabla \mathnormal{f} (\by^\ast) + \nabla^2 \mathnormal{f} (\by^\ast) \be^{(k)}\| \le T \| \be^{(k)}\|^2.
\end{equation}

In order to get an upper  bound for $\|\nabla \mathnormal{f}(\by^{(k)}) - (\bar\bJ^{(k)})^\top \bg^{(k)}\|$, recall that \linebreak
$\nabla \mathnormal{f}(\by^{(k)})= (\bJ^{(k)})^\top \bff^{(k)}$. Then,
\begin{align*}
\|\nabla \mathnormal{f}(\by^{(k)}) - (\bar\bJ^{(k)})^\top \bg^{(k)}\| &= \| (\bJ^{(k)})^\top \bff^{(k)} - (\bar\bJ^{(k)})^\top \bg^{(k)}\| \\
& \le \| (\bJ^{(k)})^\top  - (\bar\bJ^{(k)})^\top  \|   \|\bff^{(k)}\| +  \|(\bar\bJ^{(k)})^\top \| \|\bff^{(k)}- \bg^{(k)}\|.
\end{align*}
Assume that $\varepsilon^{(k)} \kappa <\frac{1}{2}$. By Lemma \ref{lem:approx_Jacobian},
\begin{align*}
\|\bJ^{(k)}  - \bar\bJ^{(k)}\| &<  8\sqrt{r(m+q)} M  \kappa_2^2(\bA_{\lambda, \bL}(\by^{(k)})) \frac{\|\bb\|}{\|\bA_{\lambda, \bL}(\by^{(k)})\|} \varepsilon^{(k)}
\le K  \varepsilon^{(k)},
\end{align*}
where $K=8\sqrt{r(m+q)} M  \kappa^2 \gamma^{-1} \| \bb\|$.
In addition, we have that
$\|\bff^{(k)}\| \le \| \bb\|$ and that inequality \eqref{eq:approx_residual} implies:
 $$\| \bff^{(k)}- \bg^{(k)} \| <  4\kappa \|\bb\| \varepsilon^{(k)}. $$
Therefore,
\begin{equation}\label{eq:gradient}
\|\nabla \mathnormal{f}(\by^{(k)}) - (\bar\bJ^{(k)})^\top \bg^{(k)}\|<  (K+ 4\kappa \alpha )\| \bb\| \varepsilon^{(k)}.
\end{equation}

To get a bound for the last term in the expression \eqref{eq: ek+1} of $\be^{(k+1)}$, note that
$$\| (\bar\bJ^{(k)})^\top \bar\bJ^{(k)} - \nabla^2 \mathnormal{f} (\by^\ast) \| \le \| (\bar\bJ^{(k)})^\top \bar\bJ^{(k)} - (\bJ^\ast)^\top \bJ^\ast\| +\|  (\bJ^\ast)^\top \bJ^\ast- \nabla^2 \mathnormal{f} (\by^\ast) \|.$$
Now, we have that
\begin{align*}
\|(\bar\bJ^{(k)})^\top \bar\bJ^{(k)} - (\bJ^\ast)^\top \bJ^\ast\|  &\le
\|(\bar\bJ^{(k)})^\top - (\bJ^\ast)^\top \| \|\bar\bJ^{(k)} \|+ \|(\bJ^\ast)^\top \| \|\bar\bJ^{(k)} - \bJ^\ast \| \\ &\le 2\alpha \|\bar\bJ^{(k)} - \bJ^\ast \|
\end{align*}
and so,
\begin{equation}\label{eq:hessian}
\| ((\bar\bJ^{(k)})^\top \bar\bJ^{(k)} - \nabla^2 \mathnormal{f} (\by^\ast) ) \be^{(k)}\| \le (2\alpha \|\bar\bJ^{(k)} - \bJ^\ast \| +\delta^\ast) \|\be^{(k)}\|.
\end{equation}

Finally, we will obtain an upper bound for $\|((\bar\bJ^{(k)})^\top \bar\bJ^{(k)})^{-1}\|$.
By Lemma \ref{lem:approx_Jacobian} and the Lipschitz assumption on $\bJ$,
$$\|\bar\bJ^{(k)} - \bJ^\ast \| \le \|\bar\bJ^{(k)} - \bJ^{(k)}\|+\|\bJ^{(k)} - \bJ^\ast \| < K  \varepsilon^{(k)} + L \| \be^{(k)}\|.$$
For a positive $\delta$,
if $\|\be^{(k)}\| \le  \frac{\delta}{2L}$, taking a tolerance $\varepsilon^{(k)}< \frac{\delta}{2 K}$, we obtain
\begin{equation}\label{eq: diffJapJmin}
\|\bar\bJ^{(k)} - \bJ^\ast \|<\delta.
\end{equation}
Then,  taking $\delta\le\frac{1}{4\alpha\beta}$, we obtain:
$$\|(\bar\bJ^{(k)})^\top \bar\bJ^{(k)} - (\bJ^\ast)^\top \bJ^\ast\|\le 2\alpha  \|\bar\bJ^{(k)} - \bJ^\ast \| < \dfrac{1}{2\beta}, $$
and, as a consequence,
$$\|((\bJ^\ast)^\top \bJ^\ast)^{-1} \left((\bar\bJ^{(k)})^\top \bar\bJ^{(k)} - (\bJ^\ast)^\top \bJ^\ast\right)\| \le 2\alpha \beta \|\bar\bJ^{(k)} - \bJ^\ast \| < \frac{1}{2}.$$
Therefore, by \cite[Theorem 3.1.4]{dennisschnabel},
\begin{align*}
\|((\bar\bJ^{(k)})^\top \bar\bJ^{(k)})^{-1}\|& \le  \frac{\|((\bJ^\ast)^\top \bJ^\ast)^{-1}\|}{1-\|((\bJ^\ast)^\top \bJ^\ast)^{-1} \left((\bar\bJ^{(k)})^\top \bar\bJ^{(k)} - (\bJ^\ast)^\top \bJ^\ast\right)\|}
< 2 \beta.
\end{align*}

From this inequality together with \eqref{eq:taylor}, \eqref{eq:gradient} and \eqref{eq:hessian}, it follows that
\begin{align*}
\| \be^{(k+1)}\|  & \le \|((\bar\bJ^{(k)})^\top \bar\bJ^{(k)})^{-1}\| \Big(T \|\be^{(k)}\|^2 +  (K  + 4\kappa \alpha )\| \bb\| \varepsilon^{(k)}   \\
& \qquad \qquad \qquad \qquad \qquad \qquad \qquad \qquad  + (2\alpha \|\bar\bJ^{(k)} - \bJ^\ast \| +\delta^\ast) \|\be^{(k)}\|\Big)\\
& < 2\beta T \|\be^{(k)}\|^2 + 2\beta (K + 4\kappa \alpha )\| \bb\| \varepsilon^{(k)} + 2\beta (2\alpha \|\bar\bJ^{(k)} - \bJ^\ast \| +\delta^\ast)\|\be^{(k)}\|.
\end{align*}

Assume $$\|\be^{(k)}\| \le \min \left\{ \frac{\delta}{2L}, \frac{1}{16 \beta T}, \frac{1}{2^{k}}\right\}.$$
Then, $2\beta T \| \be^{(k)}\|<\frac{1}{8}$ and, if $\beta\delta^* <\frac{1}{16}$, we have that
$2\beta (2\alpha \|\bar\bJ^{(k)} - \bJ^\ast \| +\delta^\ast) < \frac{1}{8}$
provided that
$\|\bar\bJ^{(k)} - \bJ^\ast \| < \frac{1-16\beta \delta^*}{32\alpha \beta}.$
Let $\delta=\frac{1-16\beta \delta^*}{32\alpha \beta}.$ Note that $\delta < \frac{1}{4\alpha\beta}$, as required by our previous bounds.
Taking into account that  $\|\bar\bJ^{(k)} - \bJ^\ast \| <\delta$ (see Equation \eqref{eq: diffJapJmin}), we conclude that
$$\| \be^{(k+1)}\| < \frac{1}{4} \|\be^{(k)} \| + 2\beta (K  + 4\kappa \alpha )\| \bb\| \varepsilon^{(k)}. $$
Now, if $\varepsilon^{(k)}\le \frac{1}{8\beta (K  + 4\kappa \alpha )\| \bb\|} \min \left\{ \frac{\delta}{L}, \frac{1}{8\beta T}, \frac{1}{2^k}\right\}$, it follows that
\begin{align*}
\| \be^{(k+1)}\| &\le \frac{1}{4}  \min \left\{ \frac{\delta}{2L}, \frac{1}{16 \beta T}, \frac{1}{2^k}\right\} +  \frac{1}{4}   \min \left\{ \frac{\delta}{L}, \frac{1}{8\beta T}, \frac{1}{2^k}\right\} \\
& \le \min \left\{ \frac{\delta}{2L}, \frac{1}{16 \beta T}, \frac{1}{2^{k+1}}\right\}.
\end{align*}
This concludes the proof.
\end{proof}

\begin{remark}\label{remark:e0} From the proof of Theorem \ref{thm: main}, we deduce how to choose the tolerances $\varepsilon^{(k)}$ to ensure the stated convergence rate for the Inexact-GenVarPro algorithm. Roughly speaking, $\varepsilon^{(k)} \simeq C/2^k$ for a constant $C$ depending on the problem. In practice, since $C$ is unknown, we choose an initial tolerance $\varepsilon^{(0)}$ satisfying  $\varepsilon^{(0)}\kappa_2(\bA_{\lambda, \bL}(\by^{(0)}))\ll 1$ so that the bound in Equation \eqref{eq:approx_solution} is positive and small, and for $k=1,2,...$, we take $\varepsilon^{(k)} = \varepsilon^{(k-1)}/2$.
\end{remark}

\section{Numerical Example}\label{sec: numexamples}
This section presents a blind deconvolution problem used to confirm our analysis. In particular, this problem was chosen so that it fits our assumptions in terms of its dimensions (small dimension for the nonlinear variable $\by$, $r=1$, and large dimension for the linear variable $\bx$, $n=128$), and we can easily compute both the exact Jacobian and its approximations.

The problem is described by the forward model
$$b(s)=\int_{a}^b g(s-s')x(s')ds',$$
where we are assuming a Gaussian kernel of the form
$$g(s)={\rm{exp}} \left( -\frac{s^2}{2\sigma^2}\right),$$
containing only the parameter $\sigma$ (i.e., in this case $\by=\sigma$). To obtain the system $\bA(\by)\bx=\bb$, we consider $n=128$ discretization points and apply the midpoint quadrature rule to approximate the integral. We assume zero boundary conditions on the function $x$ (and therefore on its discretized representation $\bx$) to obtain the symmetric Toeplitz matrix $\bA(\by)$ with its first row defined as
\begin{align}
   [\bA(\by)]_{1,j}=  c \, \exp\left(\frac{-(j-1)^2}{2\sigma^2}\right), \qquad j=1, \dots, n,\nonumber
\end{align}
where $c = 1/\left(\sum_{j} \exp(-(j-1)^2/2\sigma^2)\right)$.

\begin{figure}
    \centering
   \includegraphics[width=0.45\textwidth]{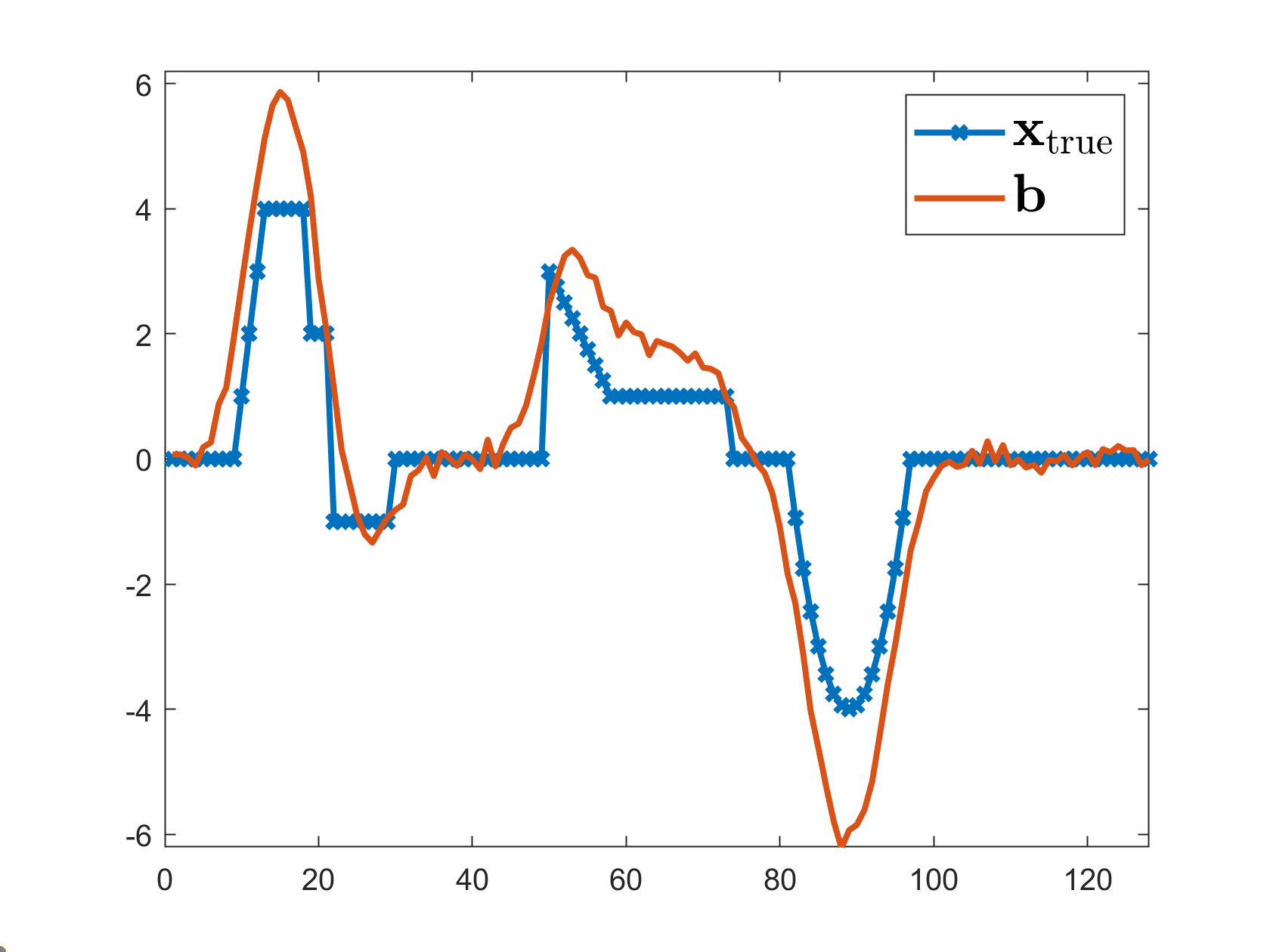}
    \includegraphics[width=0.45\textwidth]{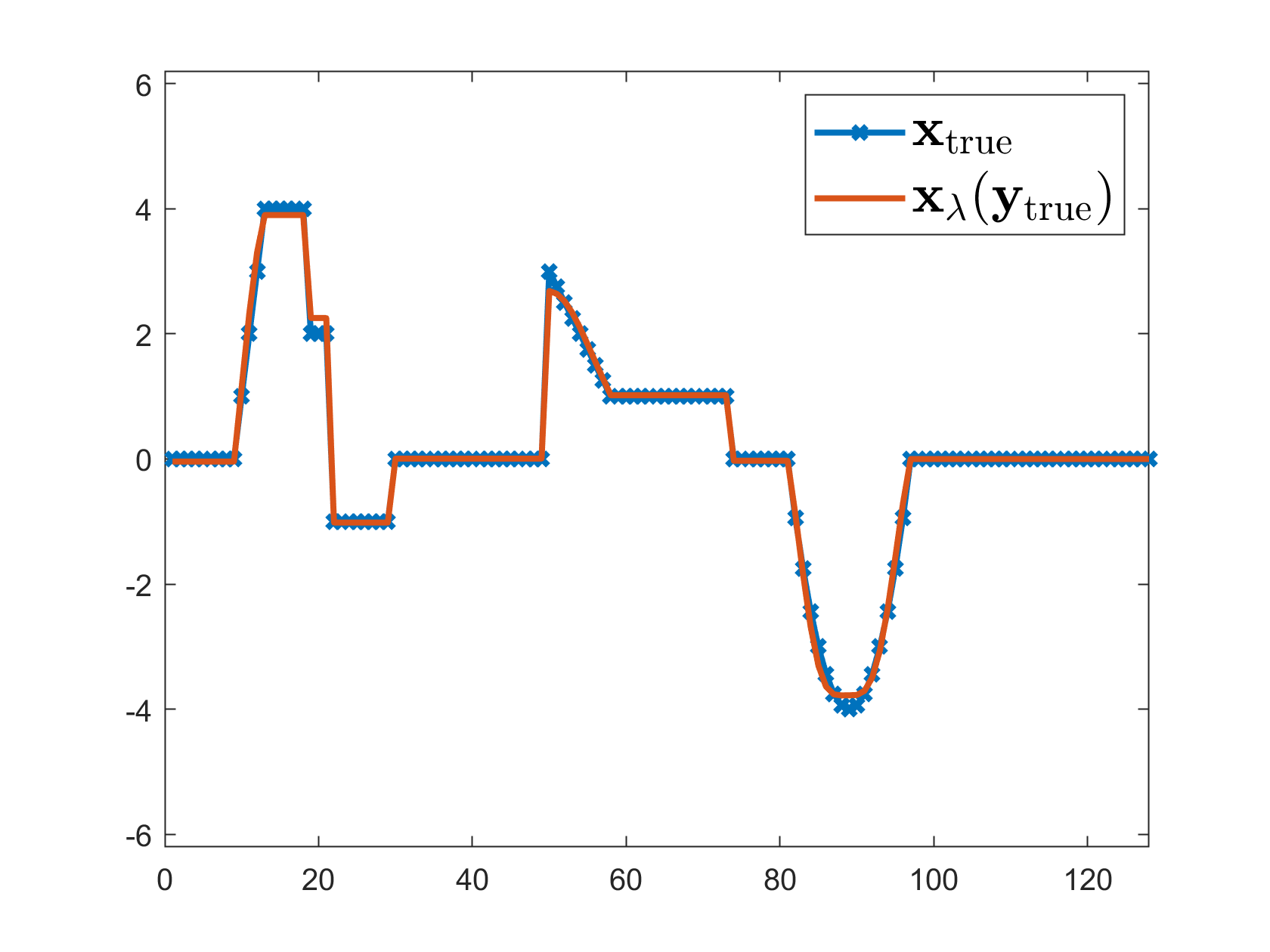}
    \caption{Left: the true signal $\bx_{{\rm true}}$ and the blurred and noisy signal $\bb$. Right: a reconstructed Tikhonov solution using $\by_{{\rm true}}$.}
    \label{fig:solution}
\end{figure}
We consider $\sigma=3$ (i.e., $\by_{{\rm true}}=3$). The condition number of $\bA(\by_{{\rm true}})$ is $1.37 \times 10^{17}$. The exact solution, represented by $\bx_{{\rm true}}$, is the vector of length $128$ shown in Figure \ref{fig:solution}. The noise-free blurred signal, represented by $\bb_{{\rm true}}$, is computed as
$\bb_{{\rm true}}=\bA(\by_{{\rm true}})\bx_{{\rm true}}$. The elements of the noise vector $\bepsilon$ are
normally distributed with zero mean, and the standard deviation is chosen such that $\frac{\|\bepsilon\|_2}{\|\bb_{{\rm true}}\|_2}=0.05$. In this case, we say that the noise level is $5\%$. The noisy right-hand side of our system is defined by $\bb=\bb_{{\rm true}}+\bepsilon$ (see Figure \ref{fig:solution}).

We choose the matrix $\bL \in \mathbb{R}^{(n-1) \times n}$ defined so that $\|\bL \bx\|_2^2 \approx \|\bD \bx\|_1$, where
\begin{align}
    \bD = \begin{bmatrix}
        -1 & 1 & 0 &\cdots && 0\\
        0 & -1 & 1 & 0 & \cdots& 0 \\
        \vdots & & \ddots & \ddots && \vdots \\
        0 & \cdots & 0 & -1 & 1 & 0 \\
        0 & \cdots & & 0& -1 & 1
    \end{bmatrix}\nonumber
\end{align}
is a discretization of the first derivative operator. Then, we define $\bL=\bW \bD$, where $\bW$ is a diagonal matrix with entries defined by $\bD\bx_{\rm true}$. We use a fixed value of $\lambda=0.0379$. This value was chosen from 20 logarithmically spaced values of $\lambda$ from $10^{-3}$ to 1 so that, together with our selection of $\bL$, we have a reduced minimization problem with minimizer $\by\approx 3$.
We solve the separable nonlinear least squares problem using GenVarPro and Inexact-GenVarPro with initial guesses $\by^{(0)}=2$ and $\by^{(0)}=4$.

To solve the linear subproblem \emph{`exactly'}, we use MATLAB backslash to solve the normal equation
$$\bA_{\lambda, \bL}^\top \bA_{\lambda, \bL}\bx = \bA_{\lambda, \bL}^\top \bb.$$ To solve the linear subproblem \emph{approximately}, we use the MATLAB built-in function LSQR. We set the maximum number of iterations for LSQR to $10,000$, which was never reached in the runs. We test four different tolerance sequences for LSQR to see the impact of approximating $\bx$ to compute the Jacobian and the residual:
\begin{itemize}
\item LSQR-b: a fixed large tolerance $\varepsilon^{(k)}=\varepsilon^{(0)}$ for every $k$;
\item LSQR-lb: a linearly decreasing tolerance $\varepsilon^{(k)}=\varepsilon^{(0)}/k$ for every $k$;
\item LSQR-ab: an exponentially decreasing tolerance $\varepsilon^{(k)}=\varepsilon^{(k-1)}/2$ for every $k$; and
\item LSQR-s: a fixed small tolerance $\varepsilon^{(k)}=10^{-11}$ for every $k$.
\end{itemize}
\medskip

Note that LSQR-ab corresponds to the tolerance sequence chosen in Algorithm \ref{Alg:I-GenVarPro}, and so, the theoretical results in Theorem \ref{thm: main} apply to this case.

Following Remark \ref{remark:e0}, we use $\varepsilon^{(0)}=1.8718\times 10^{-4}$ for $\by^{(0)}=2$ and $\varepsilon^{(0)}=1.1239\times 10^{-4}$ for $\by^{(0)}=4$ for LSQR-b, LSQR-lb, and LSQR-ab.

In \cref{fig:1dyvf}, we compare the different reconstructed parameters $\by$ at iteration $k$ given by GenVarPro and Inexact-GenVarPro with different tolerance sequences in LSQR. In that figure, we can also see how the functional $\mathcal{F}(\bar\bx^{(k)}, \by^{(k)})$ decreases. As expected, the smaller the tolerance, the faster the decrease. We can see that the convergence is independent of the initial guess $\by^{(0)}$.

\begin{figure}
    \centering
   \includegraphics[width=1\textwidth]{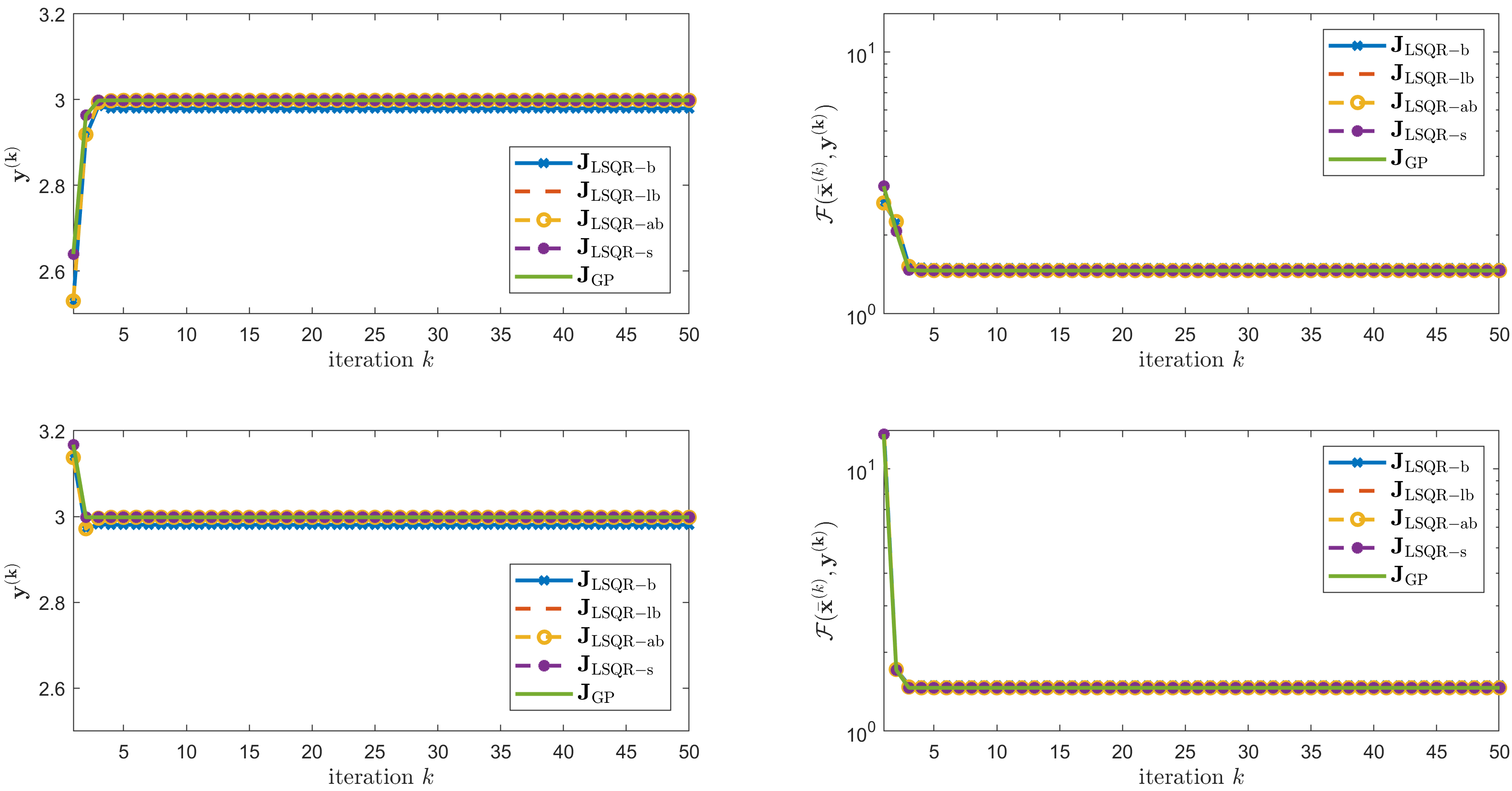}
    \caption{Convergence curves of the GenVarPro \emph{(GP)} and Inexact-GenVarPro method with different tolerances: 1) $\varepsilon^{(k)}=\varepsilon^{(0)}$ \emph{(LSQR-b)}, 2) $\varepsilon^{(k)}=\varepsilon^{(0)}/k$ \emph{(LSQR-lb)}, 3) $\varepsilon^{(k)}=\varepsilon^{(k-1)}/2$ \emph{(LSQR-ab)},  and 4) $\varepsilon^{(k)} = 10^{-11}$ \emph{(LSQR-s)}. The left column contains the values of $\by^{(k)}$ for each iteration obtained using $\by^{(0)}=2$ (top) and $\by^{(0)}=4$ (bottom). The right column depicts the values of the function $\mathcal{F}(\bar\bx^{(k)}, \by^{(k)})$ for each iteration using $\by^{(0)}=2$ (top) and $\by^{(0)}=4$ (bottom).}
    \label{fig:1dyvf}
\end{figure}

Next, in order to see the impact of the approximation of $\bx$ using LSQR, we compute the distance between the values of $\by$ at each iteration given by GenVarPro and by Inexact-GenVarPro with the different tolerance sequences. That is, we compute $|\by_{\rm GP}^{(k)}-\by_{\rm LSQR}^{(k)}|$ for $k=1,\dots, 50$ (see Figure \ref{fig:1ddify}). We use a logarithmic scale on the vertical axis on the right of Figure \ref{fig:1ddify} to better appreciate the different convergence rates. We can see there that Inexact-GenVarPro with LSQR-ab also has an exponential convergence rate, in accordance with our theoretical results.
\begin{figure}[!ht]
    \centering
   \includegraphics[width=1\textwidth]{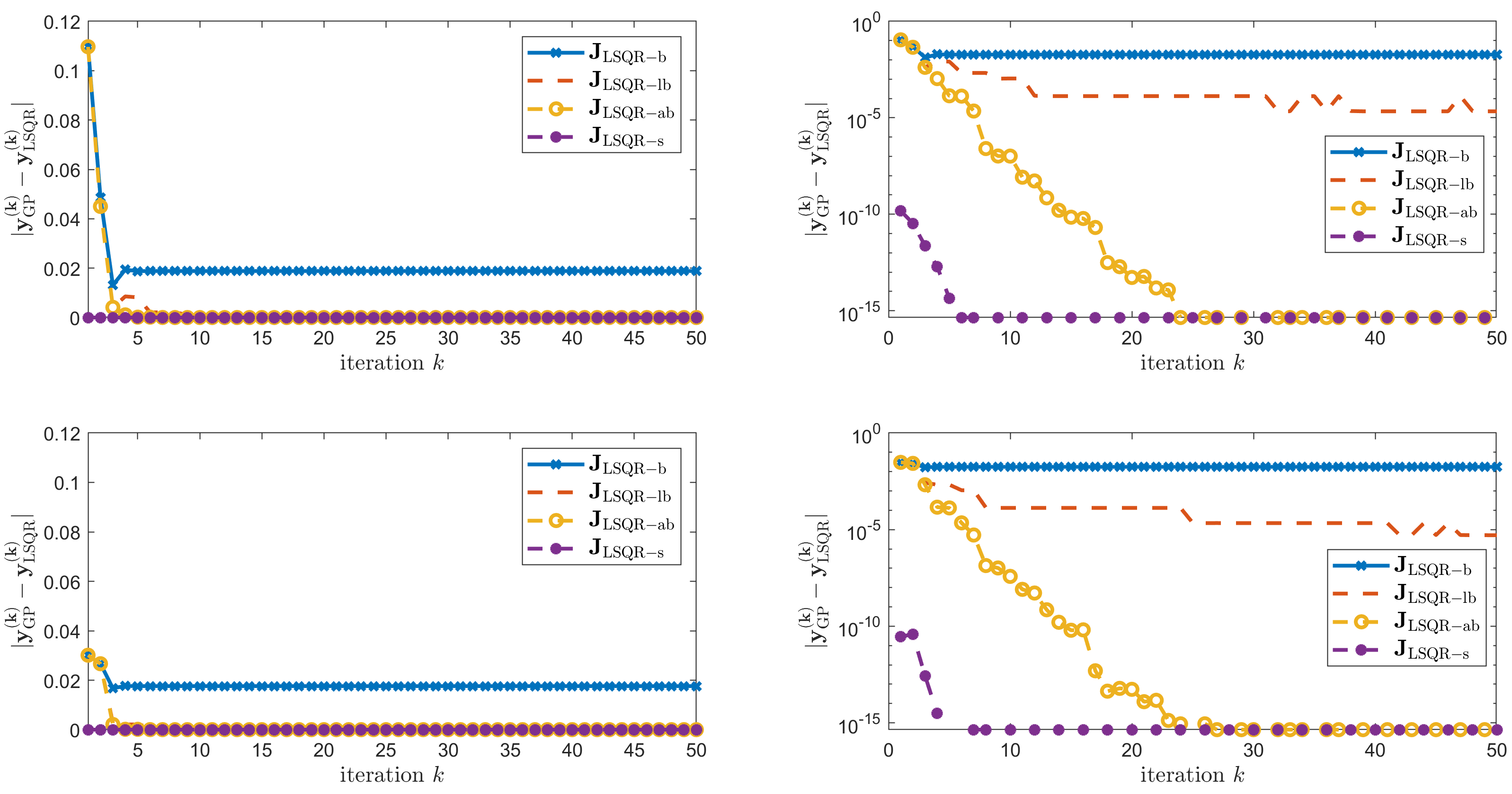}
    \caption{Distances between the solutions $\by$ given at each iteration by GenVarPro ($\by^{(k)}_{\rm{GP}}$) and Inexact-GenVarPro ($\by^{(k)}_{\rm{LSQR}}$) with different tolerances (\emph{LSQR-b}, \emph{LSQR-lb}, \emph{LSQR-ab}, and \emph{LSQR-s}) using $\by^{(0)}=2$ (top) and $\by^{(0)}=4$ (bottom). The right column contains the same data using a logarithmic scale in the vertical axis.}
    \label{fig:1ddify}
\end{figure}

\newpage
In Figure \ref{fig:1dEB}, we show $\|\bx^{(k)}-\bar\bx^{(k)}\|_2$ and the error bound $$\frac{ 2\kappa_2^2(\bA_{\lambda, \bL}(\by^{(k)}))}{1-\varepsilon\, \kappa_2(\bA_{\lambda, \bL}(\by^{(k)}))} \frac{\|\bb\|_2}{\|\bA_{\lambda, \bL}(\by^{(k)})\|_2} \varepsilon^{(k)}$$ for each iteration $k$, which appear in Equation \eqref{eq:approx_solution}. Comparing the figures on the left and the corresponding ones on the right, we see that the error bounds in \eqref{eq:approx_solution} are verified at each iteration for all tolerances, but the exponential one, which is verified up to iteration $k\approx 26$. After iteration $k=26$, some numerical issues appear, and the error bound continues decreasing while $\|\bx^{(k)}-\bar\bx^{(k)}\|_2$ stagnates. First, note that for a large $k$, the tolerance $\varepsilon^{(k)}$ in LSQR-ab reaches machine epsilon. Secondly, the $\bx^{(k)}$ is obtained using MATLAB backslash, which uses a QR algorithm, and therefore, we cannot say how accurate $\bx^{(k)}$ is.

\begin{figure}[!ht]
    \centering
   \includegraphics[width=1\textwidth]{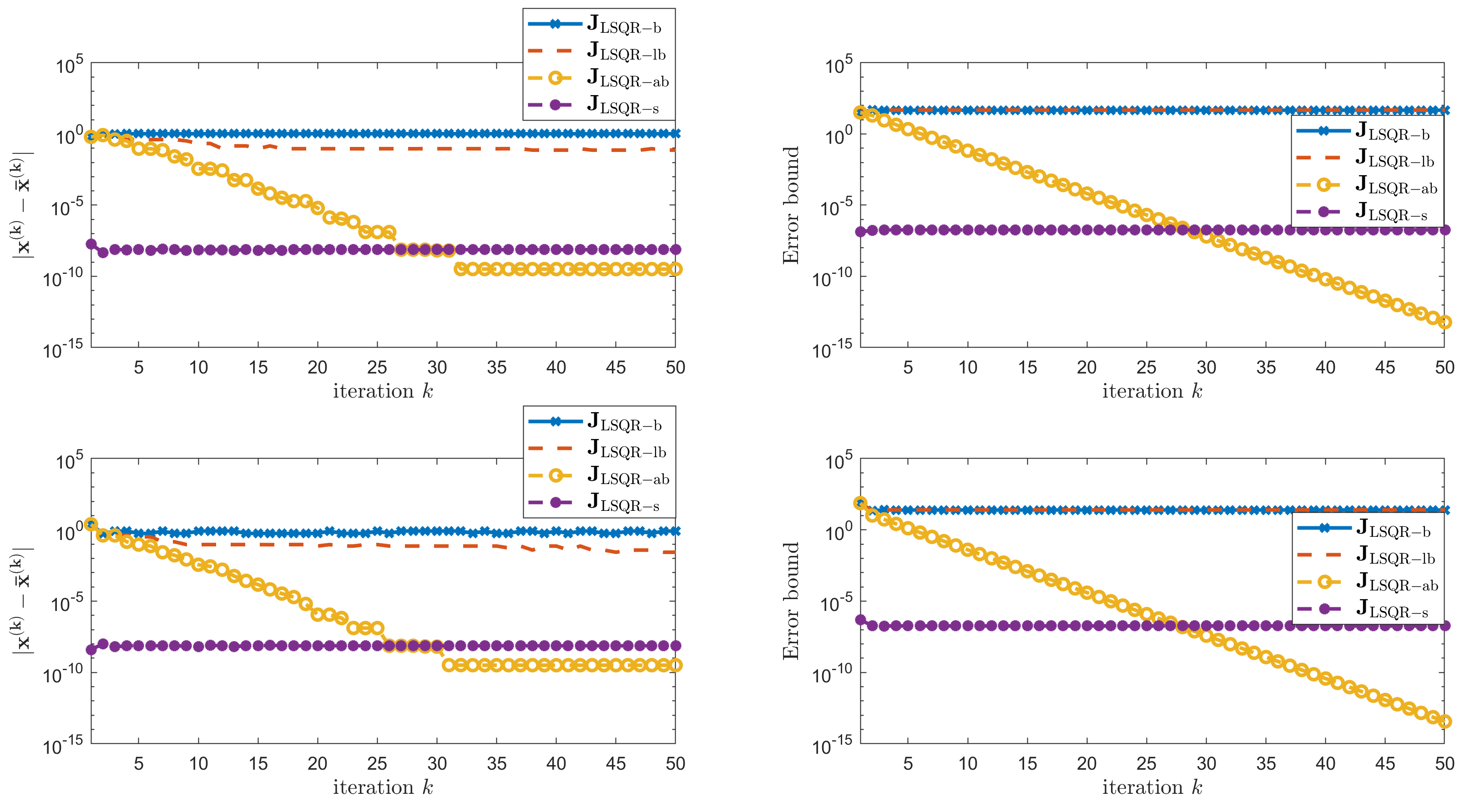}
    \caption{Left column:  $\|\bx^{(k)}-\bar\bx^{(k)}\|_2$, where $\bx^{(k)}$ is the exact solution of the linear subproblem and  $\bar\bx^{(k)}$ its LSQR approximation at each iteration, using \emph{LSQR-b}, \emph{LSQR-lb}, \emph{LSQR-ab}, and \linebreak \emph{LSQR-s}, respectively, with $\by^{(0)}=2$ (top) and $\by^{(0)}=4$ (bottom).
    Right column: error bounds computed according to \eqref{eq:approx_solution} for each case.}
    \label{fig:1dEB}
\end{figure}

All tests were performed using MATLAB R2022b on a single processor, Intel Core i9 computer. In Table \ref{tab:time}, we can see the times needed for Inexact-GenVarPro to run 26 iterations. As expected, the smaller the tolerance, the larger the time it takes.

\begin{table}[!ht]
    \centering
    \caption{Comparison of Inexact-GenVarPro with different tolerances in terms of CPU time for different initial guesses of $\by$.
    }
   \label{tab:time}
    \begin{tabular}{c|c|c|c|c}
        &  LSQR-s&  LSQR-ab & LSQR-lb & LSQR-b \\ \hline
      CPU time (seconds) - $\by^{(0)}=2$ & $1.34$ & $1.12$ & $0.78$ & $0.59$\\
      CPU time (seconds) - $\by^{(0)}=4$ & $1.21$ & $0.98$ & $0.69$ & $0.52$
    \end{tabular}
\end{table}
\newpage
To assess the quality of the reconstructed solutions given by LSQR-ab at iteration $k$, we compute the Relative Reconstruction Error (RRE) defined by
\begin{equation*}
{\rm RRE}(\bx^{(k)})=\frac{\|\bx^{(k)}-\bx_{\rm true}\|_2}{\|\bx_{\rm true}\|_2}.
\end{equation*}
Tables \ref{tab:RREy02} and \ref{tab:RREy04} compare these relative reconstruction errors and the corresponding values of $\by$ for the first seven iterations of GenVarPro and Inexact-GenVarPro using LSQR-ab, for $\by^{(0)}=2$ and $\by^{(0)}=4$, respectively. We also include the absolute values of the gradients, which decrease to 0. Figure \ref{fig:1d reconstructions} shows the vector $\bx_{\rm true}$ and the corresponding reconstruction vectors at iteration $k=7$.
\begin{figure}[h]
    \centering
   \includegraphics[width=0.45\textwidth]{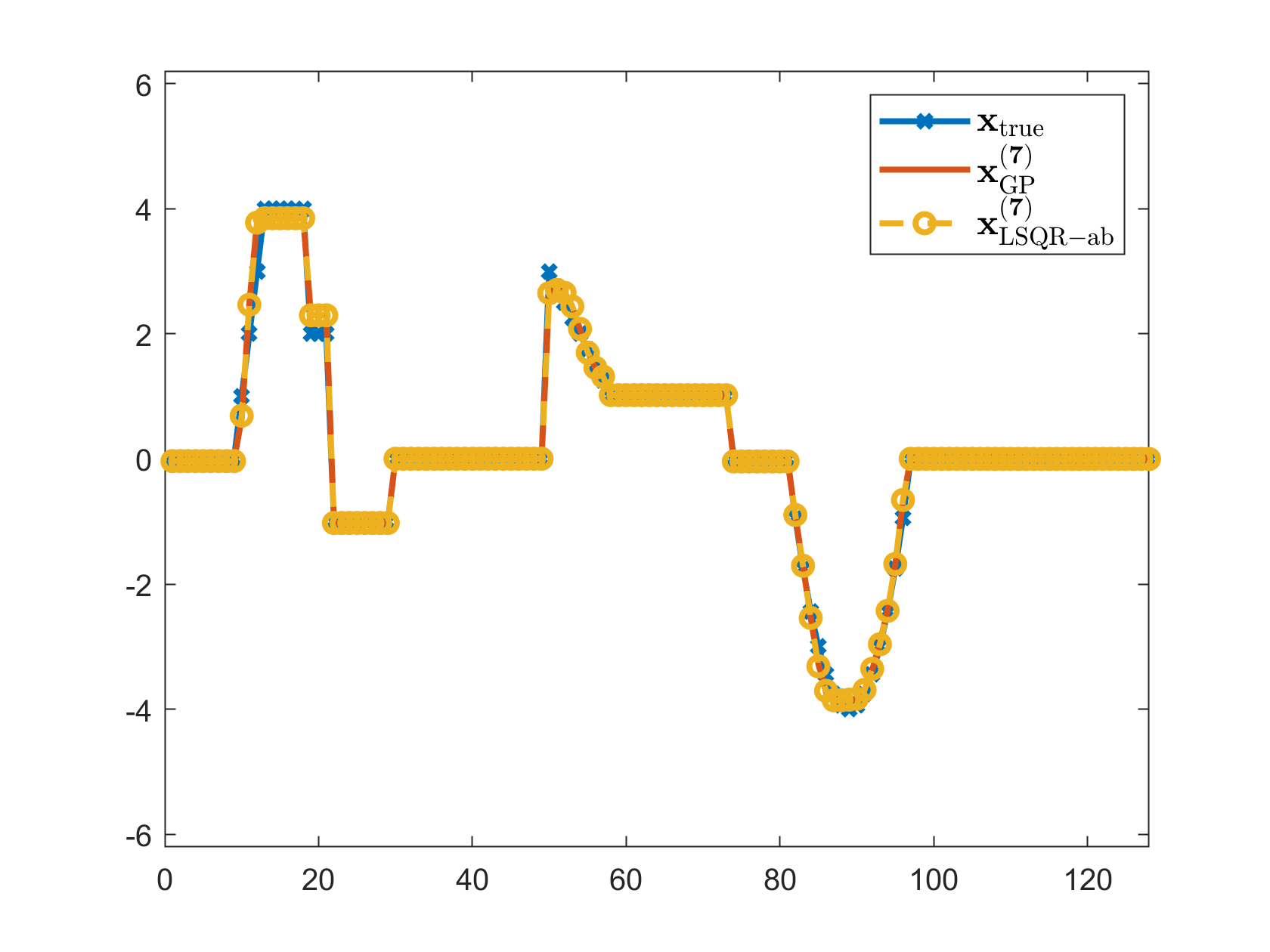}
    \includegraphics[width=0.45\textwidth]{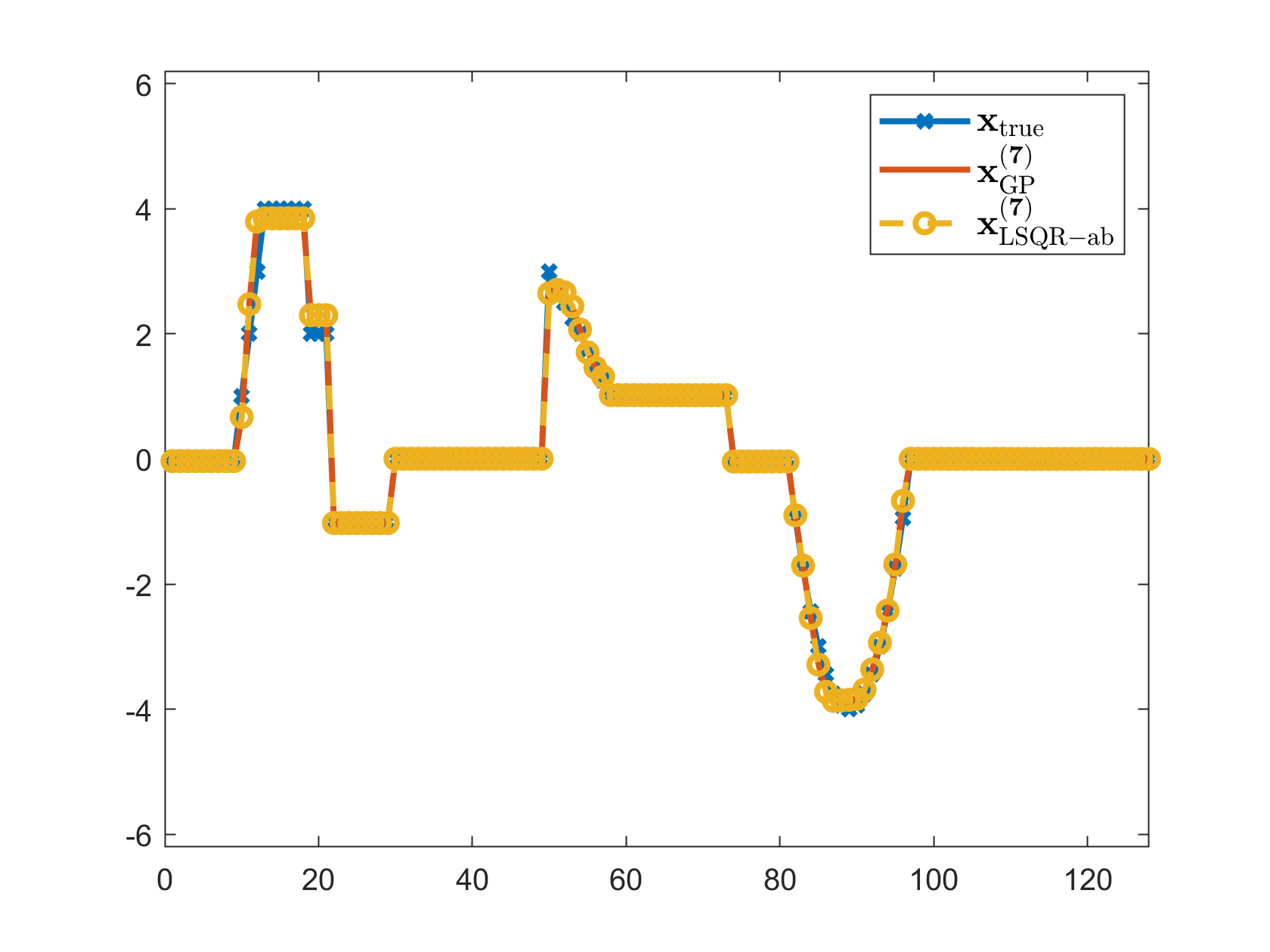}
    \caption{Reconstructions of the solution $\bx$ computed by GenVarPro and Inexact-GenVarPor using LSQR with an exponential decreasing tolerance \emph{(LSQR-ab)} for $\by^{(0)}=2$ (left) and $\by^{(0)}=4$ (right) after seven iterations.}
    \label{fig:1d reconstructions}
\end{figure}

\begin{table}[!ht]
    \centering
    \caption{Relative reconstruction errors of $\bx$, values of $\by$, and the absolute values of the gradient of the objective function obtained at each iteration of GenVarPro and Inexact-GenVarPro using LSQR with an exponential decreasing tolerance \emph{(LSQR-ab)} for $\by^{(0)}=2.$ The notation here is $\nabla f_{\rm GP}^{(k)}=\nabla f(\by_{\rm GP}^{(k)})$ and $\nabla f_{\rm LSQR-ab}^{(k)}=\nabla f(\by_{\rm LSQR-ab}^{(k)})$.
    }
   \label{tab:RREy02}
    \begin{tabular}{c|c|c|c|c|c|c}
        & RRE$(\bx_{\rm GP}^{(k)})$&  RRE$(\bx_{\rm LSQR-ab}^{(k)})$ & $\by_{\rm GP}^{(k)}$ & $\by_{\rm LSQR-ab}^{(k)}$ & $|\nabla f_{\rm GP}^{(k)}|$ & $|\nabla f_{\rm LSQR-ab}^{(k)}|$ \\ \hline
      1 & $0.2782$ & $0.1928$ & $2.6396$ & $2.5299$ & 1.6966 & 1.6966\\
      2 & $0.1087$ & $0.1154$ & $2.9635$ & $2.9185$ & 0.4793 & 0.5461\\
      3 & $0.0742$ & $0.0682$ & $2.9979$ & $2.9938$ & 0.0425 & 0.0653\\
      4 & $0.0779$ & $0.0649$ & $2.9983$ & $2.9972$ & 0.0005 & 0.0027\\
      5 & $0.0780$ & $0.0651$ & $2.9983$ & $2.9981$ & 0.0000 & 0.0002\\
      6 & $0.0780$ & $0.0651$ & $2.9983$ & $2.9981$ & 0.0000 & 0.0002\\
      7 & $0.0780$ & $0.0651$ & $2.9983$ & $2.9983$ & 0.0000 & 0.0000\\
    \end{tabular}
\end{table}
\begin{table}[h]
    \centering
    \caption{Relative reconstruction errors of $\bx$, values of $\by$, and the absolute values of the gradient of the objective function obtained at each iteration of GenVarPro and Inexact-GenVarPro using LSQR with an exponential decreasing tolerance \emph{(LSQR-ab)} for $\by^{(0)}=4.$ The notation here is $\nabla f_{\rm GP}^{(k)}=\nabla f(\by_{\rm GP}^{(k)})$ and $\nabla f_{\rm LSQR-ab}^{(k)}=\nabla f(\by_{\rm LSQR-ab}^{(k)})$.
    }

   \label{tab:RREy04}
    \begin{tabular}{c|c|c|c|c|c|c}
        & RRE$(\bx_{\rm GP}^{(k)})$&  RRE$(\bx_{\rm LSQR-ab}^{(k)})$ & $\by_{\rm GP}^{(k)}$ & $\by_{\rm LSQR-ab}^{(k)}$  & $|\nabla f_{\rm GP}^{(k)}|$ & $|\nabla f_{\rm LSQR-ab}^{(k)}|$ \\ \hline
      1 & $0.4329$ & $0.4029$ & $3.1676$ & $3.1375$ & 0.8990 & 0.8990 \\
      2 & $0.1125$ & $0.0557$ & $2.9983$ & $2.9717$ & 0.1983 & 0.1641 \\
      3 & $0.0780$ & $0.0615$ & $2.9983$ & $2.9962$ & 0.0001 & 0.0325 \\
      4 & $0.0780$ & $0.0753$ & $2.9983$ & $2.9981$ & 0.0000 & 0.0025\\
      5 & $0.0780$ & $0.0755$ & $2.9983$ & $2.9981$ & 0.0000 & 0.0002\\
      6 & $0.0780$ & $0.0766$ & $2.9983$ & $2.9983$ & 0.0000 & 0.0002\\
      7 & $0.0780$ & $0.0776$ & $2.9983$ & $2.9983$ & 0.0000 & 0.0000\\
    \end{tabular}
\end{table}

\section{Conclusions} \label{sec:conclusions}
We introduced a new variant of the variable projection method, which we call Inexact-GenVarPro, for solving large-scale separable nonlinear regularized inverse problems. In this method, we incorporated LSQR into GenVarPro (\cite{Espanol_2023}) to compute approximate solutions to the inner subproblem and used them to compute approximate Jacobians. We also proposed a stopping criterion for LSQR to ensure the convergence of our method. We presented a convergence analysis for Inexact-GenVarPro that holds not only for LSQR but also for any iterative method with the same proposed stopping criterion. Finally, we included numerical experiments where we applied Inexact-GenVarPro to solve a blind deconvolution problem. This numerical example supported our theoretical results.

In this paper, we assumed that the value of the regularization parameter $\lambda$ is known and fixed for all iterations. From our experience, having a $\lambda$ value fixed has shown overall good convergence rates~\cite{Espanol_2023}. However, it is difficult to determine a suitable value of $\lambda$ in advance. For instance, in \cite{chung2010efficient, Espanol_2023, gazzola2021regularization}, the values of $\lambda$ are updated at each iteration using some known heuristics for linear problems.

Future work includes exploring the efficacy of the proposed method with the use of iterative methods beyond LSQR (such as those introduced in~\cite{KilmerHansenEspanol, lampe2012large}), incorporating a selection method for the regularization parameter, and analyzing the addition of a regularization term for the nonlinear variables.

\section*{Acknowledgments} M.I. Espa\~nol was supported through a Karen Uhlenbeck EDGE Fellowship. Part of this work was done while G. Jeronimo was visiting Arizona State University in May 2023 and January-February 2024. We are grateful for the support and stimulating atmosphere provided by this institution.


\end{document}